\documentclass[11pt,oneside,a4paper,reqno]{amsart}

\usepackage[utf8]{inputenc}
\usepackage[T1]{fontenc}
\usepackage[english]{babel}

\usepackage{lmodern}

\usepackage[
    activate={true,nocompatibility},
    final,
    tracking=true,
    spacing=true,
    stretch=15,
    shrink=15
]{microtype}
\microtypecontext{spacing=nonfrench}
\SetTracking{encoding=*,family=*}{-20}

\usepackage[
    a4paper,
    margin=1.1in,
    top=1.2in,
    bottom=1.2in,
    headheight=15pt,
    headsep=0.4in,
    footskip=1.2cm
]{geometry}
\usepackage[dvipsnames,svgnames]{xcolor}
\usepackage{enumerate}
\definecolor{academicblue}{HTML}{1A365D}
\definecolor{winfull}{HTML}{7B2D3A}
\definecolor{wintrade}{HTML}{C99AA2}
\definecolor{hlred}{HTML}{9B2C2C}
\definecolor{rowgray}{gray}{0.95}

\usepackage{amsmath,amssymb,amsthm}
\usepackage{bm}

\usepackage{hyperref}
\hypersetup{
    colorlinks=true,
    citecolor=academicblue,
    linkcolor=academicblue,
    urlcolor=academicblue,
    breaklinks=true
}

\numberwithin{equation}{section}

\theoremstyle{plain}
\newtheorem{theorem}{Theorem}[section]
\newtheorem{proposition}[theorem]{Proposition}
\newtheorem{lemma}[theorem]{Lemma}

\theoremstyle{definition}

\newtheorem{assumption}[theorem]{Assumption}

\theoremstyle{remark}
\newtheorem{remark}[theorem]{Remark}

\newcommand{\RR}{\mathbb{R}}
\newcommand{\NN}{\mathbb{N}}

\newcommand{\Om}{\Omega}
\newcommand{\jM}{\mathbb{S}^{2}}

\newcommand{\dx}{\,\mathrm{d}x}

\newcommand{\nn}{\bm{n}}
\newcommand{\vv}{\bm{v}}
\newcommand{\gv}{\bm{g}}
\newcommand{\EE}{\bm{E}}
\newcommand{\DD}{\bm{D}}

\newcommand{\nuvec}{\bm{\nu}}
\newcommand{\gam}{\bm{\gamma}}
\newcommand{\eee}{\bm{\varepsilon}}

\newcommand{\Aeps}{{\bm{A}}_{\varepsilon}}

\newcommand{\xis}{\xi_{\star}}

\newcommand{\unn}{u_{\nn}}
\newcommand{\vnn}{v_{\nn}}

\newcommand{\ns}{\nn_{\ast}}
\newcommand{\nse}{\nn_{\ast}^{\varepsilon}}

\newcommand{\eps}{\varepsilon}

\newcommand{\vpar}{\vv_{\parallel}}
\newcommand{\vperp}{\vv_{\perp}}

\newcommand{\divg}{\operatorname{div}}
\newcommand{\curl}{\operatorname{\mathbf{curl}}}

\newcommand{\Hcurl}{H(\curl;\Om)}
\newcommand{\Hdiv}{H(\divg;\Om)}

\newcommand{\XN}{X_{\mathrm{N}}(\Om)}
\newcommand{\Vspace}{V(\Om)}

\makeatletter

\begin{document}

\title[Variational principles for liquid crystals and electric fields]{Variational principles for the interaction of liquid crystals and electric fields in the Oseen--Frank model}

\author{Giovanni Di Fratta}
\address{Dipartimento di Matematica e Applicazioni ``R. Caccioppoli'', Universit\`a degli Studi di Napoli ``Federico II'', Via Cintia, 80126 Napoli, Italy.}
\email{giovanni.difratta@unina.it}

\author{Valeriy V. Slastikov}
\address{School of Mathematics, University of Bristol, Bristol BS8 1TW, United Kingdom.}
\email{valeriy.slastikov@bristol.ac.uk}

\author{Arghir D. Zarnescu}
\address{BCAM, Basque Center for Applied Mathematics, Mazarredo 14, E48009 Bilbao, Bizkaia, Spain; IKERBASQUE, Basque Foundation for Science, Maria Diaz de Haro 3, 48013 Bilbao, Bizkaia, Spain; Simion Stoilow Institute of Mathematics of the Romanian Academy, P.O. Box 1-764, 014700 Bucharest, Romania.}
\email{azarnescu@bcamath.org}


\begin{abstract}
We develop a rigorous variational framework for uniaxial nematic liquid crystals interacting with an external electric field in the one-constant Oseen--Frank approximation. Equilibrium configurations are governed by a nonlocal, nonlinear energy with a challenging min-max saddle-point structure. Our first main result reformulates this problem as a pure double-minimization problem. Using convex duality and the Hodge decomposition, we replace the scalar electrostatic potential with a vector potential, yielding a closed-form dual functional with a unique minimizer. This direct energy-minimization principle is advantageous for both theoretical analysis and numerical simulation. Our second result rigorously quantifies the decoupling of the electrostatic back-reaction in the limit of small dielectric anisotropy. We establish a uniform quadratic energy bound between the exact nonlocal energy and its standard local approximation, formally justifying the widespread physics heuristic of neglecting induced depolarization fields. Finally, under a strict coercivity condition, we combine $\Gamma$-convergence, uniform Sobolev regularity, and a perturbative coercivity transfer to prove that physical minimizers converge to limiting harmonic maps at an optimal, quantitative linear rate.
\end{abstract}

\maketitle

\section{Introduction}\label{sec:intro}

Nematic liquid crystals are a class of soft materials that combine the fluidity of isotropic liquids with the orientational order of solid crystals \cite{degennes1993, Chaikin1995}. At the microscopic level, they consist of rod-like molecules whose locally preferred alignment is represented mathematically by a unit-length vector field, the \emph{director}. A defining macroscopic feature of nematics---and the primary mechanism enabling their use in display devices---is their characteristic electro-optic response \cite{virga1994, stewart2004}.

Within the continuum Oseen--Frank theory, the interaction between the nematic director and an externally applied electric field is dictated by the material's dielectric anisotropy, $\varepsilon_a = \varepsilon_{\parallel} - \varepsilon_{\perp}$. Here, $\varepsilon_{\parallel}$ and $\varepsilon_{\perp}$ denote the dielectric permittivities measured parallel and perpendicular to the director, respectively. When subjected to a voltage, the director field experiences a dielectric torque that induces reorientation, as observed in phenomena like the Fréedericksz transition \cite{Arakelyan1984, gartland2021}. This reorientation dynamically alters the local dielectric permittivity tensor of the medium. According to Maxwell's equations---specifically Gauss's law for dielectrics---this spatial heterogeneity in the permittivity induces a depolarization field that back-reacts on the original electric potential \cite{gartland2020}. Consequently, the true physical equilibrium of the system is governed by a fully coupled electromechanical energy functional that is nonlinear and intrinsically \emph{nonlocal}. From a variational perspective, this coupling manifests as a saddle-point (min--max) problem: the total free energy is minimized with respect to the director field's elastic distortions, but it must be maximized with respect to the induced electrostatic potential.

To simplify the analytical and computational treatment of this system, it is a standard approach in the physics literature to employ a \emph{local} approximation \cite{virga1994, gartland2020}. By assuming the external electric field is unaffected by the liquid crystal's reorientation, one decouples the potential and obtains a purely local energy functional. Physically, this approximation is formally justified in the limit of small dimensionless dielectric anisotropy, $\varepsilon = \varepsilon_a / \varepsilon_{\perp} \ll 1$. In this regime, the feedback effect of the director on the electric field is considered negligible \cite[Section 4.1.1]{virga1994}. However, establishing the mathematical validity of this limit, deriving quantitative rates of convergence, and understanding its impact on ground-state selection present distinct challenges. The standard direct methods of the calculus of variations do not apply directly to the min--max formulation, and the topology of harmonic maps into the sphere \cite{Hardt1986, BethuelBrezisCoron90, Riviere95} can lead to non-uniqueness and bifurcation phenomena.

In this paper, we develop a variational framework to study this coupled electro-nematic system and present two main contributions. First, we resolve the min--max saddle-point structure by reformulating the system as a double-minimization problem. Combining Fenchel--Rockafellar convex duality \cite{Temam1999} with the $L^2$ Hodge decomposition \cite{Girault2012}, we introduce a divergence-free vector potential that yields a coercive, closed-form dual functional. This minimization principle provides an alternative foundation for theoretical analysis and numerical methods \cite{Hu2009}. Second, we provide a quantitative justification for the local approximation in the small anisotropy regime ($\varepsilon \to 0$). By isolating the energetic gap and employing a perturbative coercivity transfer from limiting harmonic maps, we prove that physical minimizers converge to the local approximation at an optimal linear rate. Finally, we discuss a topological bifurcation phenomenon, showing that the exact physical energy and its local approximation can accumulate to macroscopically distinct ground states.

\subsection{The variational setting}\label{ssec:varprob}In the Oseen--Frank
theory, the elastic free energy of a nematic in a bounded Lipschitz domain
$\Om \subset \RR^3$ is $\int_{\Om} W_e (\widetilde{\nn}) \dx$ with
\begin{equation}
  \label{eq:OFenergy} 2 W_e (\widetilde{\nn}) = K_1 (\divg \widetilde{\nn})^2 + K_2  (\widetilde{\nn} \cdot \curl
  \widetilde{\nn})^2 + K_3  | \widetilde{\nn} \times \curl \widetilde{\nn} |^2,
\end{equation}
the constants $K_1, K_2, K_3 > 0$ penalizing the splay, twist, and bend modes
of distortion \cite{stewart2004,virga1994}. We work throughout under the
\emph{one-constant approximation} $K_1 = K_2 = K_3 = K$, in which the
elastic density reduces to the Dirichlet energy density $\tfrac{K}{2} | \nabla
\widetilde{\nn} |^2$.

The interaction with an external electric field $\widetilde{\EE} $ is governed by Maxwell's equations which reduce to Gauss--Faraday laws
\begin{equation}
  \widetilde{\DD}=\eee
(\widetilde{\nn}) \widetilde{\EE}, \qquad
  \mathrm{div} \hspace{0.17em} \widetilde{\DD}= \rho, \qquad \mathbf{curl}
  \,\widetilde{\EE}= 0 \hspace{0.27em} \hspace{0.27em} \text{in } \Om,
  \label{eq:FGews} \quad
\end{equation}
where $\rho$ is the free-charge density. The last equation provides that $\widetilde{\EE}=-\nabla \tilde{U}$ for some potential $\tilde{U}$.

The constitutive relation $\widetilde{\DD} = \eee
(\nn) \widetilde{\EE}$ with governs the interaction between the nematic liquid crystal $n$ and
an electric field $\widetilde{\EE}$ through the {\emph{dielectric permittivity
tensor}} $\eee(\nn)$, which reflects the
anisotropic nature of the medium. We have
\begin{equation}
  \label{eq:permittivity} \eee (\widetilde{\nn}) = \varepsilon_0  \left(
  \varepsilon_{\perp} I + \varepsilon_a \hspace{0.17em} \widetilde{\nn} \otimes \widetilde{\nn}
  \right), \qquad \varepsilon_a := \varepsilon_{\|} - \varepsilon_{\perp},
\end{equation}
where $\varepsilon_0$ is the vacuum permittivity and $\varepsilon_{\perp},
\varepsilon_{\|} > 0$ are the relative permittivities respectively perpendicular and
parallel to $\widetilde{\nn}$ \cite{degennes1993,vertogen1988}.

The total free energy of a nematic liquid crystal in the presence of an electric field is given by the sum of two contributions:
\begin{equation} \label{eq:ennandE}
  \widetilde{\mathcal{E}} (\widetilde{\nn}, \widetilde{\EE}) = \int_{\Om} W_e (\widetilde{\nn})
  + W_{\widetilde{\EE}}(\widetilde{\nn}, \widetilde{\EE}) \, \mathrm{d} x,
\end{equation}
where $W_e (\widetilde{\nn})$ denotes, as before, the elastic energy associated with the distortions in the molecular orientation, and $W_{\widetilde{\EE}}(\widetilde{\nn},
\widetilde{\EE})$ represents the energy due to the interaction between the liquid crystal and the applied electric field $\widetilde{\EE}$ given by: 
\begin{equation}
  W_{\widetilde{\EE}} (\widetilde{\nn}, \widetilde{\EE}) = - \frac{1}{2} \widetilde{\DD} \cdot
  \widetilde{\EE}, \label{eq:freeendenDE}
\end{equation}
 
 We consider the non-dimensionalization described in Appendix~\ref{sec:nondim} and from now on we work in a non-dimensional setting that we describe subsequently. We start by denoting
\begin{equation}
  \eps := \frac{\varepsilon_a}{\varepsilon_{\perp}} > - 1
\end{equation}
to be the dimensionless \emph{dielectric anisotropy} parameter, and define
the rescaled dielectric tensor
\begin{equation} \label{eq:Aeps}
   \Aeps (\nn) := I + \eps \hspace{0.17em} \nn \otimes \nn .
\end{equation}
This is a symmetric matrix-valued $L^{\infty}$ field whose eigenvalue is $1$
with multiplicity two on $\nn^{\perp}$, and $1 + \eps$ with multiplicity one
along $\nn$. Using the Pythagorean identity $| \zeta |^2 = | \zeta \times \nn
|^2 + (\zeta \cdot \nn)^2$, we obtain
\begin{equation}
  \label{eq:pointwiseeps} \Aeps (\nn) \zeta \cdot \zeta \hspace{0.27em} =
  \hspace{0.27em} | \zeta \times \nn |^2 + (1 + \eps)  (\zeta \cdot \nn)^2,
  \qquad \forall \hspace{0.17em} \zeta \in \RR^3,
\end{equation}
which, for $\zeta \neq 0$, shows that $| \zeta |^{- 2} \Aeps (\nn) \zeta
\cdot \zeta$ is a convex combination of $1$ and $1 + \eps$. Consequently, it is
bounded and uniformly elliptic in the sense that
\begin{equation}
  \label{eq:ellipticity} \min (1, 1 + \eps) \hspace{0.17em} | \zeta |^2
  \hspace{0.27em} \leqslant \hspace{0.27em} \Aeps (\nn) \zeta \cdot \zeta
  \hspace{0.27em} \leqslant \hspace{0.27em} \max (1, 1 + \eps) \hspace{0.17em}
  | \zeta |^2, \qquad \forall \hspace{0.17em} \zeta \in \RR^3,
\end{equation}
as long as\footnote{It should be noted that this requirement amounts to $\varepsilon_{\parallel}>0$ which is always the case.} $\varepsilon > - 1$.

The (non-dimensional) electrostatic potential $\unn \in H^1 (\Om)$ associated
with $\nn$ is the unique weak solution of
\begin{equation}
  \label{eq:Poisson} \divg \left( \Aeps (\nn) \nabla \unn \right) = 0 \quad
  \text{in } \Om, \qquad \unn = \xi_b  \quad \text{on } \partial \Om,
\end{equation}
for prescribed boundary potential $\xi_b \in H^{1 / 2} (\partial \Om)$; the
map $\nn \mapsto \unn$ is well-defined and continuous by
\eqref{eq:ellipticity}, but genuinely \emph{nonlinear and nonlocal} in
$\nn$. Let $\xis \in H^1 (\Om)$ denote the harmonic extension of $\xi_b$,
i.e., the unique solution of
\begin{equation}
  \label{eq:harmext} \Delta \xis = 0 \quad \text{in } \Om, \qquad \xis = \xi_b
  \quad \text{on } \partial \Om ;
\end{equation}
 Note that this is the canonical reference potential, corresponding to $\unn$ at $\eps =
0$ or, more generally, to the case in which the dielectric coupling is
neglected. 
For technical reasons, we assume throughout this work that the boundary datum
$\xi_b$ admits a globally Lipschitz harmonic extension, namely $\xis \in W^{1,
\infty} \left( \Om \right)$. This assumption is standard and provides a
minimal regularity requirement for the second-order asymptotic expansion
carried out in the proof of Theorem~\ref{thm:thm2}, where terms of the form $\| \ensuremath{\boldsymbol{v}} \cdot \nabla \xis \|_{L^{2}}$ arise with $\nn \cdot \ensuremath{\boldsymbol{v}}= 0$ and $\ensuremath{\boldsymbol{v}}
\in H_0^1 \left( \Om \right)$. By classical Schauder theory, this condition is
automatically satisfied, for instance, when both the domain $\Om$ and the
boundary datum $\xi_b$ are of class $C^{1, \alpha}$, which implies $\xis \in
C^{1, \alpha} \left( \Om \right)$. For nonsmooth domains, global Lipschitz
regularity of the harmonic extension $\xis$ still holds provided that $\Om$ is
convex and $\xi_b \in W^{1, \infty} \left( \partial \Om \right)$.

Imposing a strong-anchoring condition $\nn = \nn_b$ on $\partial \Om$ for some
prescribed $\nn_b \in H^{1} (\Om ; \jM)$, the admissible class is
\begin{equation}
  \label{eq:admissible} \mathcal{A}_{\nn_b} (\Om) := \{ \nn \in H^1 (\Om, \jM)
  : \nn = \nn_b \text{ on } \partial \Om \},
\end{equation}
and the (dimensionless) equilibrium configurations are minimizers, over
$\mathcal{A}_{\nn_b} (\Om)$, of
\begin{equation}
  \label{eq:mainenergy} \mathcal{E}_\varepsilon (\nn) \hspace{0.27em} := \hspace{0.27em}
  \int_{\Om} \left[ \hspace{0.17em} \alpha \hspace{0.17em} | \nabla \nn |^2 -
  \Aeps (\nn) \nabla \unn \cdot \nabla \unn \hspace{0.17em} \right]  \dx,
  \qquad \unn \text{ solving \eqref{eq:Poisson},}
\end{equation}
where $\alpha > 0$ is the dimensionless ratio of elastic to electrostatic
energy (Appendix~\ref{sec:nondim}).

Two structural features drive the entire
analysis. The functional is \emph{nonconvex} in $\nn$ both through the
unit-length constraint $| \nn | = 1$ and through the quadratic-in-$\nn$
coupling with the electrostatic field; and \emph{nonlocal} in $\nn$ through
the solution operator of \eqref{eq:Poisson}. Together these features place
\eqref{eq:mainenergy} outside the scope of classical direct-method results for
$\jM$-valued harmonic maps; the latter constitute the special case $\eps = 0$.

A further conceptual obstacle is the \emph{saddle-point} character of the
energy. Decomposing $\unn = v + \xis$ with $v \in H^1_0 (\Om)$, the
homogeneous component $v$ is the unique solution of
\begin{equation}
  \label{eq:vsolve} \divg \left( \Aeps (\nn) \nabla v \right) = - \divg \left(
  \Aeps (\nn) \nabla \xis \right) \quad \text{in } \Om, \qquad v \in H^1_0
  (\Om) ;
\end{equation}
which we denote  by $\vnn$, so that $\unn = \vnn + \xis$. From the variational
perspective, $\vnn$ can be characterized as the \emph{unique maximizer} in
$H_0^1 \left( \Om \right)$ of the homogeneous electrostatic functional
\begin{equation}
  \label{eq:WE} \mathcal{W}_E (\nn, v) := - \hspace{-0.17em} \int_{\Om} \Aeps
  (\nn) \nabla v \cdot \nabla v \dx - 2 \hspace{-0.17em} \int_{\Om} \Aeps
  (\nn) \nabla v \cdot \nabla \xis \dx, \qquad v \in H^1_0 (\Om) .
\end{equation}
Expanding the quadratic at $\unn = \vnn + \xis$ and using $\mathcal{W}_E (\nn,
\vnn) = \max_{v \in H_0^1 \left( \Om \right)} \mathcal{W}_E (\nn, v)$, the
energy \eqref{eq:mainenergy} factorizes as
\begin{equation}
  \label{eq:Ereduction} \mathcal{E}_\varepsilon (\nn) \hspace{0.27em} = \hspace{0.27em}
  \alpha \hspace{-0.17em} \int_{\Om} | \nabla \nn |^2 \dx \hspace{0.27em} -
  \hspace{0.27em} \int_{\Om} \Aeps (\nn) \nabla \xis \cdot \nabla \xis \dx
  \hspace{0.27em} + \hspace{0.27em} \max_{v \in H^1_0 (\Om)} \mathcal{W}_E
  (\nn, v) .
\end{equation}
The middle term is independent of $v$ but \emph{depends on $\nn$}: more
precisely, by \eqref{eq:Aeps}, it splits as the true constant $-
\hspace{-0.17em} \int | \nabla \xis |^2$ plus the $\nn$-dependent piece $-
\eps \hspace{-0.17em} \int (\nn \cdot \nabla \xis)^2 \dx$. Pulling the entire
middle term inside the max, minimizing \eqref{eq:mainenergy} is therefore
equivalent to the min--max problem
\begin{equation}
  \label{eq:minmax} \min_{\nn \in \mathcal{A}_{\nn_b} (\Om)}  \hspace{0.27em}
  \max_{v \in H^1_0 (\Om)} \left[ \hspace{0.17em} \alpha \hspace{-0.17em}
  \int_{\Om} | \nabla \nn |^2 \dx - \int_{\Om} \Aeps (\nn) \nabla \xis \cdot
  \nabla \xis \dx +\mathcal{W}_E (\nn, v) \right],
\end{equation}
the identification $\min \mathcal{E}_\varepsilon= \min_{\nn} \max_v (\cdot)$ being exact,
with no additive constants discarded. The order of operations matters: one
must first maximize in $v$ (yielding $\unn = \vnn + \xis$) and only then
minimize in $\nn$. Direct minimization in both variables would yield $-
\infty$ in $v$.

\subsection{Main results}\label{ssec:results}The first main result of this
paper converts the min--max structure of~\eqref{eq:minmax} into a genuine
\emph{min--min problem} by introducing a vector potential as the dual
variable. This is not merely cosmetic: minimizing simultaneously in $(\nn,
\gam)$ allows one to apply direct methods of the calculus of variations to
both variables and brings the electric-coupling term into a form that is local
and uniformly bounded below. It is also crucial for a large class of numerical simulations that use gradient descent schemes to identify local minimizers. Having a formulation in terms of minimisation only (as opposed to a min-max one) allows one to access this type of schemes directly. 

In what follows, $\Hcurl$ denotes the space of $L^2  (\Om, \RR^3)$ vector
fields with $L^2$ curl, and $\XN$ denotes the Coulomb-gauge subspace
\begin{equation}
  \label{eq:XN} \XN := \{ \gam \in \Hcurl : \divg \gam = 0 \text{ in } \Om,
  \gam \cdot \nuvec = 0 \text{ on } \partial \Om \} .
\end{equation}
Also, we denote by $\Vspace$ the subspace of $L^2 (\Om, \RR^3)$ consisting of
divergence-free vector fields:
\begin{equation}
  \label{eq:Vspace} \Vspace := \{ \gv \in L^2 (\Om, \RR^3) : \divg \gv = 0
  \text{ in } \Om \} .
\end{equation}

\begin{theorem}[Min--min reformulation via duality]
  \label{thm:thm1} Let $\Om \subset \RR^3$ be a bounded, simply connected
  Lipschitz domain with connected boundary, let $\xis$ be the harmonic
  extension \eqref{eq:harmext} of $\xi_b \in H^{1 / 2} (\partial \Om)$, and
  let $\vnn \in H^1_0 (\Om)$ be the unique solution of \eqref{eq:vsolve}.
  Define the dual functional
  \begin{equation}
    \label{eq:Adual} \mathcal{A}_{\eps} (\nn, \gv) \hspace{0.27em} :=
    \hspace{0.27em} \int_{\Om} \Aeps^{- 1} \left( \nn \right)  \left( \gv +
    \Aeps (\nn) \nabla \xis \right) \cdot \left( \gv + \Aeps (\nn) \nabla \xis
    \right) \dx, \qquad \gv \in L^2  (\Om, \RR^3) .
  \end{equation}

  Then
  \begin{align}
    \mathcal{W}_E (\nn, \vnn) = \max_{v \in H^1_0 (\Om)} \mathcal{W}_E (\nn,
    v) & = \min_{\gv \in \Vspace} \mathcal{A}_{\eps} (\nn, \gv) \nonumber\\
    & = \min_{\gam \in \Hcurl} \mathcal{A}_{\eps} (\nn, \curl \gam) =
    \min_{\gam \in \XN} \mathcal{A}_{\eps} (\nn, \curl \gam), 
    \label{eq:thm1main}
  \end{align}
  the minimum over $\Vspace$ being attained at the unique $\gv_{\ast} = -
  \Aeps (\nn) \nabla \unn$, and the minimum over $\XN$ at a \emph{unique}
  $\gam \in \XN$ satisfying
  \begin{equation}
    \label{eq:curlgamma} \curl \gam = - \Aeps (\nn) \nabla \unn .
  \end{equation}
  If, in addition, $\partial \Om$ is of class $C^{1, 1}$ or $\Om$ is convex,
  then $\gam \in H^1 (\Om, \RR^3)$ and
  \begin{equation}
    \label{eq:thm1H1} \mathcal{W}_E (\nn, \vnn) \hspace{0.27em} =
    \hspace{0.27em} \min_{\gam \in H^1 (\Om, \RR^3)} \mathcal{A}_{\eps} (\nn,
    \curl \gam) \hspace{0.27em} = \hspace{0.27em} \min_{\gam \in H^1 (\Om,
    \RR^3) \cap \XN} \mathcal{A}_{\eps} (\nn, \curl \gam) .
  \end{equation}
\end{theorem}

Applying Theorem~\ref{thm:thm1} to the inner maximization in
\eqref{eq:minmax}, problem \eqref{eq:mainenergy} is equivalent to the min--min
problem
\begin{equation}
  \label{eq:minmin} \min_{\nn \in \mathcal{A}_{\nn_b} (\Om)}  \hspace{0.27em}
  \min_{\gam \in \Hcurl} \left[ \hspace{0.17em} \alpha \hspace{-0.17em}
  \int_{\Om} | \nabla \nn |^2 \dx - \int_{\Om} \Aeps (\nn) \nabla \xis \cdot
  \nabla \xis \dx +\mathcal{A}_{\eps} (\nn, \curl \gam) \right],
\end{equation}
again with exact identification of the minimum values. The gain over
\eqref{eq:minmax} is conceptual: the inner problem is now a minimization, the
joint functional is uniformly bounded below for each fixed $\xis$, and the
coupling with $\nn$ enters through the localized weight $\Aeps^{- 1} (\nn)$
rather than through the global solution operator of \eqref{eq:Poisson}. The
additional term $- \hspace{-0.17em} \int \Aeps (\nn) \nabla \xis \cdot \nabla
\xis$, local in $\nn$, becomes the boundary-extension contribution that drives
the quantitative asymptotic limit of Theorem~\ref{thm:thm2}.

The second main result quantifies the regime in which the electrostatic
coupling effectively decouples, justifying a formal asymptotic expansion
present in the physical literature \cite[§4.1.1]{virga1994}. Physically,
the dimensionless anisotropy parameter is
\begin{equation}
  \label{eq:eps} \eps = \frac{\varepsilon_a}{\varepsilon_{\perp}}
\end{equation}
(see Appendix~\ref{sec:nondim}). The relevant physical regime is $\varepsilon
\to 0$, where the dielectric tensor $\Aeps (\nn) = I + \varepsilon
\hspace{0.17em} \nn \otimes \nn$ degenerates to the identity $I$, and the
electrostatic potential $\unn$ approaches the harmonic extension $\xis$.

To make this rigorous, we introduce two distinct functionals: the \emph{local} functional $\mathcal{G}_{\varepsilon}$ and the \emph{coupled}
functional $\mathcal{F}_{\varepsilon}$:
\begin{align}
  \mathcal{G}_{\eps} (\nn) & := \alpha \hspace{-0.17em} \int_{\Om} | \nabla
  \nn |^2 \dx - \int_{\Om} \Aeps (\nn) \nabla \xis \cdot \nabla \xis \dx,
  \qquad \nn \in \mathcal{A}_{\nn_b} (\Om),  \label{eq:Geps}\\
  \mathcal{F}_{\eps} (\nn, \gv) & := \mathcal{G}_{\eps} (\nn)
  +\mathcal{A}_{\eps} (\nn, \gv), \qquad (\nn, \gv) \in \mathcal{A}_{\nn_b}
  (\Om) \times \Vspace .  \label{eq:Feps}
\end{align}
Here, $\mathcal{A}_{\nn_b} (\Om)$ and $\Vspace$ are defined in
\eqref{eq:admissible} and \eqref{eq:Vspace}, respectively. The unperturbed
base energy (for $\varepsilon=0$) is therefore given by:
\begin{equation}
  \mathcal{G}_0 (\nn) = \alpha \int_{\Om} | \nabla \nn |^2 \dx - \int_{\Om} |
  \nabla \xis |^2 \dx . \label{eq:G0fun}
\end{equation}
These functionals have distinct, complementary roles. The local functional
$\mathcal{G}_{\varepsilon} (\nn)$ isolates the harmonic-map contribution; it
is the energy of the system if one entirely ignores the dielectric
back-reaction. The coupled functional $\mathcal{F}_{\varepsilon} (\nn, \gv)$
is the duality lift of the full nonlocal energy $\mathcal{E}_\varepsilon$ in the dual
variable $\gv$. By Theorem~\ref{thm:thm1}, we know that
\begin{equation}
  \label{eq:Fmin} \min_{\gv \in \Vspace} \mathcal{F}_{\eps} (\nn, \gv)
  \hspace{0.27em} = \hspace{0.27em} \mathcal{E}_\varepsilon (\nn) .
\end{equation}
Because the decomposition $\mathcal{F}_{\varepsilon}
=\mathcal{G}_{\varepsilon} +\mathcal{A}_{\varepsilon}$ cleanly separates the
local energy from the non-negative electrostatic correction
$\mathcal{A}_{\varepsilon} \geqslant 0$, the maximum error introduced by
approximating the full coupled energy $\mathcal{E}_\varepsilon (\nn)$ with the local
functional $\mathcal{G}_{\varepsilon} (\nn)$ is exactly:
\begin{equation}
  \label{eq:error_def} \sup_{\nn \in H^1 (\Om, \jM)} \left| \mathcal{E}_\varepsilon(\nn)
  -\mathcal{G}_{\varepsilon} (\nn) \right| = \sup_{\nn \in H^1 (\Om, \jM)}
  \max_{v \in H^1_0 (\Om)} \mathcal{W}_E (\nn, v) = \sup_{\nn \in H^1 (\Om,
  \jM)} \min_{\gv \in \Vspace} \mathcal{A}_{\varepsilon} (\nn, \gv) .
\end{equation}
To convert this energy gap into a quantitative $H^1$-norm estimate for the
minimizers, we require the second variation of the base energy $\mathcal{G}_0$
to be strictly coercive. This coercivity guarantees uniform quadratic growth
of the energy away from the minimum, providing the necessary analytical
control to bound the distance between the minimizers of the coupled and
uncoupled problems.

\begin{assumption}
  \textbf{(Coercivity at the limit)}\label{ass:coercivity} A global
  minimizer $\ns \in H^1 (\Om, \jM)$ of $\mathcal{G}_0$ satisfies the
  coercivity condition
  \begin{equation}
    \label{eq:coerc} \mathcal{L}_0 [\vv] := \int_{\Om} | \nabla \vv |^2 \dx -
    \int_{\Om} | \vv |^2 | \nabla \ns |^2 \dx \geqslant \alpha_0 \| \vv
    \|_{H^1_0 (\Om)}^2
  \end{equation}
  for every $\vv \in H^1_0 (\Om, \RR^3)$ with $\vv \cdot \ns = 0$ a.e. in
  $\Om$, for some constant $\alpha_0 > 0$.
\end{assumption}

We provide sufficient conditions for Assumption~\ref{ass:coercivity} in
Appendix~\ref{ssec:coercivity-suff}: it holds universally when the boundary
data takes values in an open hemisphere of $\jM$ (Lemma~\ref{lem:hemisphere}),
and can be verified explicitly in the parallel-plate geometry relevant to
physical display devices (Lemma~\ref{lem:parallelplates}). We now state our main quantitative convergence theorem.

\begin{theorem}[Rates of convergence to the asymptotic limit]
  \label{thm:thm2}Let $\Om \subset \RR^3$ be a bounded Lipschitz domain. Then
  there exists a constant $C_{\xi} > 0$, depending only on $\xis$ and $\Om$,
  such that
  \begin{equation}
    \label{eq:energybound} \sup_{\nn \in H^1 (\Om, \jM)} \left| \min_{\gv \in
    \Vspace} \mathcal{E}_\varepsilon(\nn) -\mathcal{G}_{\varepsilon}
    (\nn) \right| \hspace{0.27em} \leqslant \hspace{0.27em} C_{\xi} 
    \hspace{0.17em} \varepsilon^2  \qquad \text{for every } \varepsilon \in
    (0, 1) .
  \end{equation}
  Suppose, in addition, that $\partial \Om$ is of class $C^3$ and $\nn_b \in
  C^{\infty}  (\partial \Om ; \jM)$. Then for every family $\{ \nse
  \}_{\varepsilon \to 0}$ of global minimizers of $\mathcal{G}_{\varepsilon}$
  in $\mathcal{A}_{\nn_b} (\Om)$ there exist a subsequence (not relabeled) and
  a global minimizer $\ns$ of $\mathcal{G}_0$ in $\mathcal{A}_{\nn_b} (\Om)$
  such that:\smallskip
  \begin{enumerate}[(i)]
    \item $\| \nse - \ns \|_{H^1 (\Om)} \to 0$ as $\varepsilon \to 0$;\smallskip
    
    \item if, in addition, $\ns$ satisfies Assumption~\ref{ass:coercivity},
    then there exist constants $K_{\xi}, \varepsilon_{\bullet} > 0$ such that,
    for every family $\{\bm{m}^{\varepsilon}\} \subset \arg \min \mathcal{E}_{\varepsilon}$
    with $\bm{m}^{\varepsilon} \to \ns$ in $H^1 (\Om,
    \RR^3)$,
    \begin{equation}
      \label{eq:rate} \| \nse - \bm{m}^{\varepsilon}
      \|_{H^1 (\Om)} \hspace{0.27em} \leqslant \hspace{0.27em} K_{\xi} 
      \hspace{0.17em} \varepsilon \qquad \text{for every } \varepsilon \in (0,
      \varepsilon_{\bullet}) .
    \end{equation}
  \end{enumerate}
\end{theorem}

{\begin{remark}
  \label{rmk:bifurcation}Statement (ii) quantifies the error incurred when an
  exact physical minimizer $\ensuremath{\boldsymbol{m}}^{\varepsilon} \in \arg
  \min \mathcal{E}_{\varepsilon}$ is replaced by its local approximate
  minimizer $\ensuremath{\boldsymbol{n}}^{\varepsilon} \in \arg \min
  \mathcal{G}_{\varepsilon}$. Although both families of functionals
  $\Gamma$-converges to $\mathcal{G}_0$ (see Section~\ref{ssec:gamma}), such a comparison is meaningful only when the corresponding minimizing sequences
  converge to the same harmonic map $\ensuremath{\boldsymbol{n}}^{\ast} \in
  \mathrm{argmin} \mathcal{G}_0$. This assumption is essential. Indeed, owing
  to the intrinsic non-convexity of the problem, minimizers of $\mathcal{G}_0$
  need not be unique, and there is no {{\em a priori\/}} reason to expect
  arbitrary minimizing sequences of $\mathcal{E}_{\varepsilon}$ and
  $\mathcal{G}_{\varepsilon}$ will select the same asymptotic limit. As shown
  in Appendix~\ref{app:bifurcation}, even in physically relevant
  geometries such as a slab (or a rod), there exist families of minimizers
  $\{\ensuremath{\boldsymbol{m}}^{\varepsilon} \} \subset \arg \min
  \mathcal{E}_{\varepsilon}$ and $\{ \nse \} \subset \arg \min
  \mathcal{G}_{\varepsilon}$ that bifurcate as $\varepsilon \rightarrow 0$:
  \begin{equation}
    \ensuremath{\boldsymbol{m}}^{\varepsilon} \to \ensuremath{\boldsymbol{m}}
    \quad \text{and} \quad \nse \to \ensuremath{\boldsymbol{n}}_{\ast} \qquad
    \text{strongly in } H^1 (\Omega),
  \end{equation}
  where $\ensuremath{\boldsymbol{m}}, \ensuremath{\boldsymbol{n}}_{\ast} \in
  \arg \min \mathcal{G}_0$ are distinct limiting ground states, i.e.,
  $\ensuremath{\boldsymbol{m}} \neq \ensuremath{\boldsymbol{n}}_{\ast}$.
\end{remark}

\begin{remark}It should be noted that working directly with the local functional $\mathcal{G}_{\eps}$, as usually done in the physics literature, can potentially produce the {\it ``wrong" physical minimizer}. Indeed, one could be in the situation as noted above where the minimizers $\nse$ are such that  $\quad \nse \to \ensuremath{\boldsymbol{n}}_{\ast}$ but the minimizers of the full energy functional $\ensuremath{\boldsymbol{m}}^{\varepsilon}$, the more physically relevant ones, {\it are not close to $\nse$}.
\end{remark}

\begin{remark}[Regularity considerations and choice of the reference family]
  When formulating the comparison between the two functionals, one could theoretically center the analysis on the minimizers of either $\mathcal{E}_{\varepsilon}$ or $\mathcal{G}_{\varepsilon}$. In our framework, we choose to treat the local approximation $\mathcal{G}_{\varepsilon}$ as the reference family. This choice is primarily motivated by the analytical properties of the associated Euler--Lagrange equations.  Specifically, the local functional $\mathcal{G}_{\varepsilon}$ readily admits uniform $W^{2,2}$ regularity bounds for its minimizers because the background electric potential $\xi_b$ inherently belongs to $W^{1,\infty}(\Omega)$. Moreover, the functional $\mathcal{G}_\varepsilon
 $ is the standard choice in the physics literature, because of its simplicity.  
  
  Conversely, establishing the same uniform regularity for the exact physical energy $\mathcal{E}_{\varepsilon}$ presents significant technical obstacles. The minimizers of $\mathcal{E}_{\varepsilon}$ depend non-locally on the induced depolarization potential $v$. For a generic director field $\boldsymbol{n} \in H^1(\Omega; \mathbb{S}^2)$, standard elliptic theory (via Meyers-type estimates) only guarantees $v \in W^{1,p}(\Omega)$ for some undetermined $p > 2$. Since this generally falls short of the $W^{1,\infty}$ bound required to systematically bootstrap uniform $W^{2,2}$ estimates for the director field, treating $\mathcal{G}_{\varepsilon}$ as the regular base sequence and $\mathcal{E}_{\varepsilon}$ as the perturbation bypasses this regularity bottleneck and yields a cleaner, more robust mathematical statement.
\end{remark}
}

Two structural features of Theorem~\ref{thm:thm2} elucidate the mechanism underlying the electro-nematic coupling. First, the $O(\varepsilon^2)$ scaling in  the energy bound \eqref{eq:energybound} arises from a precise algebraic cancellation built into the dual formulation of $\mathcal{F}_{\varepsilon}$. More specifically, choosing the background field $\gv = -\nabla \xis \in \Vspace$ as a competitor yields an exact balance between the divergence-free constraint and the boundary-extension terms. By contrast, a straightforward estimate based solely on algebraic bounds would produce only the suboptimal scaling $O(\varepsilon)$.

Second, translating this $O (\varepsilon^2)$ energy gap into the optimal $O
(\varepsilon)$ strong $H^1$-norm estimate in \eqref{eq:rate} relies entirely
on the quadratic coercivity \eqref{eq:coerc} of the second variation. This
analytical control must be robustly transferred from the uncoupled ground
state to the fully coupled $\varepsilon$-system.

\begin{remark}[The necessity of coercivity]
  \label{rem:coercivity}While the uniform energy bound \eqref{eq:energybound}
  holds unconditionally, Assumption~\ref{ass:coercivity} is strictly necessary
  to deduce the convergence rate. Without strict coercivity at the limit, the
  energy landscape may possess flat directions, causing minimizers of the
  coupled system to converge at suboptimal rates or even fail to select a
  unique branch of solutions. The technical core of Theorem~\ref{thm:thm2}
  lies in transferring this coercivity from $\ns$ to the $\varepsilon$-coupled
  family $\nse$ (Lemma~\ref{lem:coerctransfer}). This transfer is the
  principal motivation for establishing uniform $W^{2, 2}$ regularity in
  Theorem~\ref{thm:uniformW22est}: securing strong convergence of $\nse \to
  \ns$ in $W^{1, 4} (\Om)$ provides the exact analytical integrability
  required to reliably perturb the coercivity inequality.
\end{remark}

\subsection{Strategy of proof}\label{ssec:strategy}Theorem~\ref{thm:thm1} is
proved by Fenchel--Rockafellar duality adapted to the geometric setting of
$L^2$ vector fields. Writing $\mathcal{W}_E (\nn, v) = - \varphi (v) - \psi
(v)$ with $\varphi$ the convex electrostatic energy and $\psi$ a linear
source, the Fenchel--Rockafellar theorem identifies $\max \mathcal{W}_E$ with
the conjugate minimization $\min_{f \in H^{- 1}} (\varphi^{\ast} (- f) +
\psi^{\ast} (f))$. The natural parametrization $f = - \divg \bm{h}$ and the $L^2$ Hodge decomposition reduce the dual
variable to a divergence-free vector field $\gv \in \Vspace$, and
Theorem~\ref{thm:isoCurl} (the isomorphism of the curl on simply connected
Lipschitz domains) lifts $\gv$ to a unique vector potential $\gam \in \XN$.
Although the duality route is conceptually motivating, the resulting formula
\eqref{eq:Adual} admits a direct quadratic-completion verification, which we
present in Section~\ref{sec:duality}.

Theorem~\ref{thm:thm2} proceeds in three steps, each separating a distinct
mechanism:

\smallskip
\noindent\emph{Step 1: Energy comparison.} Using $\gv = - \nabla \xis$ as a competitor in
$\Vspace$ and exploiting the algebraic identity $I - \tfrac{\varepsilon}{1 +
\varepsilon} \hspace{0.17em} \nn \otimes \nn = (I + \varepsilon
\hspace{0.17em} \nn \otimes \nn)^{- 1}$, we obtain matching two-sided bounds
yielding \eqref{eq:energybound}. The estimate is uniform in $\nn \in H^1 (\Om,
\jM)$ and depends only on $\| \nabla \xis \|_{L^{\infty}}$ and $| \Om |$.

\smallskip
\noindent\emph{Step 2: $\Gamma$-convergence.} The functionals $\mathcal{G}_{\varepsilon}$
$\Gamma$-converge to $\mathcal{G}_0$ in the $L^2$ topology, and the unit-norm
constraint forces the perturbation $- \varepsilon \hspace{-0.17em} \int (\nn
\cdot \nabla \xis)^2$ to vanish uniformly. Strong $H^1$-convergence of
minimizers follows by the standard energy-coincidence argument.

\smallskip
\noindent\emph{Step 3: Quantitative rate.} Uniform $W^{2, 2}$-regularity of the minimizers
(Theorem~\ref{thm:uniformW22est}) upgrades $H^1$-convergence to $W^{1,
4}$-convergence, which is precisely what is needed to perturb the
second-variation coercivity at $\ns$ to the entire family $\nse$
(Lemma~\ref{lem:coerctransfer}). The energy gap of order $\varepsilon^2$ then
converts to an $H^1$-bound of order $\varepsilon$ via the quadratic
coercivity, yielding \eqref{eq:rate}.

The uniform regularity rests on the partial regularity theory of Hardt,
Kinderlehrer, and Lin~\cite{Hardt1986} and the difference-quotient estimates
of Borchers and Garber~\cite{Borchers}, both adapted to the
$\varepsilon$-perturbed Euler--Lagrange equation and to the boundary
flattening permitted by the $C^3$ regularity of $\partial \Om$.

\subsection{Discussion and related literature}\label{ssec:literature} The interaction of nematic liquid crystals with electric and magnetic fields has been extensively studied in the physics literature, with foundational accounts provided in {\cite{degennes1993,virga1994,stewart2004,sonnet2012}}. By contrast, while the mathematical theory of the uncoupled Oseen--Frank model is highly developed {\cite{Ball2017}}, the rigorous analysis of systems driven by nonlinear and nonlocal field interactions remains comparatively underexplored.

Early analytical investigations focused on the existence and partial regularity of coupled electromechanical equilibria. In their seminal work {\cite{Hardt1986}}, Hardt, Kinderlehrer, and Lin established the existence of minimizers in the presence of dielectric effects. By extending classical harmonic-map techniques {\cite{BCL,GW,schoen_uhlenbeck_1983,Si}} to the variable-coefficient energies induced by the dielectric tensor, they demonstrated that minimizers are real-analytic outside a singular set of zero one-dimensional Hausdorff measure.

However, partial regularity does not resolve the global structure or uniqueness of minimizers, which are strongly dictated by the topology of the target manifold $\mathbb{S}^2$. The phenomenon of non-uniqueness, including the existence of infinitely many weakly harmonic maps for nonconstant boundary data, was established in \cite{BethuelBrezisCoron90,Riviere95} and subsequently extended to anisotropic Oseen--Frank energies in \cite{Hong2007}. In the context of electric coupling, we show that this severe non-uniqueness introduces a delicate variational vulnerability: different approximation schemes may cause minimizing sequences to bifurcate toward macroscopically distinct limiting configurations (see the explicit construction in Appendix~\ref{app:bifurcation}).

This multiplicity of states is intimately tied to field-induced bifurcations and macroscopic instabilities. The variational consistency of electrically coupled nematic models was recently clarified in {\cite{gartland2020}}. A prototypical instability mechanism is the Fréedericksz transition, characterized by a sudden reorientation of the director field once the external voltage exceeds a critical threshold. Rigorous stability criteria for this transition in two dimensions demonstrate that the exact nonlinear and nonlocal dielectric coupling strictly raises the critical threshold for destabilization {\cite{gartland2021}}. 

Beyond the primary transition, linearized analyses predict secondary instabilities, such as the formation of periodic stripe domains in nematic films under higher field strengths {\cite{Allender1987}}. Because these analyses typically neglect the nonlocal electrostatic back-reaction, their mathematical validity has remained heuristic. The optimal convergence rate established in \eqref{eq:rate} now places this localized approximation on a rigorous footing in the small-anisotropy limit.

The inherent mathematical complexity of this nonlocal coupling can be further illuminated by contrasting the Oseen--Frank system with the Landau--Lifshitz model of micromagnetics \cite{Di_Fratta_2020}. Both theories describe orientational order via a unit-vector field penalized by a Dirichlet-type energy and coupled to a nonlocal scalar potential via a saddle-point formulation. In micromagnetics, however, the medium is magnetically homogeneous. The magnetization $\bm{m}$ acts solely as a linear divergence source, meaning the magnetostatic potential satisfies the constant-coefficient Poisson equation $\Delta U = \operatorname{div}\bm{m}$ and admits an explicit representation via Newtonian convolution.
In stark contrast, a nematic liquid crystal structurally defines its own dielectric environment. The director $\bm{n}$ dictates the anisotropy of the medium, forcing the electrostatic potential to solve a variable-coefficient equation $\operatorname{div}(\bm{\varepsilon}(\bm{n})\nabla u)=0$. Because the dielectric tensor $\bm{\varepsilon}(\bm{n})$ depends nonlinearly on $\bm{n}$, the coupling is fully nonlinear. Consequently, exact integral convolution representations are unavailable, fundamentally altering the functional framework required for both analysis and numerical simulation (for a detailed comparative discussion, see the supplementary material of~\cite{Gartland_2015}).

These structural impediments highlight the analytical features of our model. In particular, two aspects distinguish the present contribution.

First, the duality formula \eqref{eq:Adual} appears to be new in this context. Existing approaches invariably rely on the min--max formulation \eqref{eq:minmax} or its associated Euler--Lagrange system. Our vector-potential reformulation eliminates the saddle-point bottleneck, placing both variables within a strictly coercive, double-minimization structure perfectly suited for direct methods of the calculus of variations and modern discretization. A structurally related, though physically distinct, saddle-point formulation for harmonic maps was introduced in {\cite{Hu2009}}, where the dual variable acts as a Lagrange multiplier enforcing the unit-length constraint rather than representing the physical electric displacement.

Second, the quantitative limit
\eqref{eq:energybound}--\eqref{eq:rate} provides a rigorous counterpart to
previously established formal expansions {\cite[§4.1.1]{virga1994}}. It
identifies both the algebraic structure responsible for the $\varepsilon^2$
accuracy of the local approximation and the precise role of the
second-variation coercivity in transferring this accuracy to minimizers.
Analogous coercivity-transfer mechanisms have been exploited in Landau--de
Gennes theory {\cite{Majumdar2010,Nguyen2013,Fratta2020}} and in the
ferronematic setting {\cite{Canevari2023}}. The argument presented here adapts
the technical toolkit of uniform $W^{2, p}$ estimates and perturbative
second-variation analysis to the genuinely nonlocal coupling
\eqref{eq:Poisson} that defines the Oseen--Frank--electrostatic problem.

\subsection{Organization of the paper}Section~\ref{sec:duality} is devoted to
the proof of Theorem~\ref{thm:thm1}; Section~\ref{sec:asymptotic} to
Theorem~\ref{thm:thm2}. Appendix~\ref{sec:nondim} records the
non-dimensionalization. Appendix~\ref{sec:regularity} gathers the uniform
$W^{2, 2}$-regularity result for the perturbed minimizers and verifies
Assumption~\ref{ass:coercivity} in two geometric configurations.

\subsection*{Notation}Throughout, $\Om \subset \RR^3$ is a bounded Lipschitz
domain, simply connected with connected boundary unless stated otherwise. We
write $L^2  (\Om, \RR^3)$ for the space of square-integrable vector fields,
with norm $\| \cdot \|_{L^2}$. The spaces $\Hdiv$ and $\Hcurl$ are the
$L^2$-completions under the graph norms of the divergence and curl operators,
respectively. The gauge space $\XN$ is defined in \eqref{eq:XN}, and $\Vspace$
in \eqref{eq:Vspace}. The outward unit normal to $\partial \Om$ is denoted
$\nuvec$. Constants are denoted $C$ and may change from line to line; their
dependence on the data is indicated by subscripts.

\section{Min--min reformulation: proof of
Theorem~\ref{thm:thm1}}\label{sec:duality}

The proof exploits two ingredients: a quadratic completion that transforms the
electrostatic maximization into a perfect square, and the surjectivity of the
curl operator on simply connected Lipschitz domains. Throughout this section,
$\nn \in H^1 (\Om, \jM)$ is fixed, and we abbreviate $A := \Aeps (\nn) = I +
\eps \hspace{0.17em} \nn \otimes \nn$, a uniformly positive symmetric
matrix-valued $L^{\infty}$ field. By the pointwise identity
\eqref{eq:pointwiseeps},
\begin{equation}
  \label{eq:Apos} \min (1, 1 + \eps) \hspace{0.17em} | \zeta |^2
  \hspace{0.27em} \leqslant \hspace{0.27em} A \zeta \cdot \zeta
  \hspace{0.27em} \leqslant \hspace{0.27em} \max (1, 1 + \eps) \hspace{0.17em}
  | \zeta |^2, \qquad \forall \hspace{0.17em} \zeta \in \RR^3 .
\end{equation}

\subsection{Functional setting and isomorphism of the
curl}\label{ssec:setting}We recall the functional analytic framework. For $\gv
\in \Hdiv$, the normal trace $\gv \cdot \nuvec \in H^{- 1 / 2} (\partial \Om)$
is defined by
\begin{equation}
  \label{eq:normaltrace} \langle \gv \cdot \nuvec, \phi \rangle_{\partial \Om}
  := \int_{\Om} \gv \cdot \nabla \phi \dx + \int_{\Om} (\divg \gv) 
  \hspace{0.17em} \phi \dx, \qquad \phi \in H^1 (\Om) ;
\end{equation}
see, e.g., Girault and Raviart~\cite[Thm.~I.2.5]{Girault2012}. The
Coulomb-gauge space $\XN$ is defined in \eqref{eq:XN}; the divergence-free
space $\Vspace$ in \eqref{eq:Vspace}. When $\Om$ is simply connected with
connected boundary, $\divg \gv = 0$ in $\Om$ automatically implies $\langle
\gv \cdot \nuvec, 1 \rangle_{\partial \Om} = 0$ via the divergence theorem, so
no explicit boundary flux condition is needed in the definition of $\Vspace$.

The cornerstone of our duality argument is the following classical result.

\begin{theorem}[Isomorphism of the curl,
\cite{Girault2012}]
  \label{thm:isoCurl}Let $\Om \subset \RR^3$ be a bounded, simply connected
  Lipschitz domain with connected boundary. The operator
  \[ \curl : \XN \to \Vspace \]
  is a topological isomorphism. Explicitly, for every $\gv \in \Vspace$ there
  exists a unique $\gam \in \XN$ such that $\curl \gam = \gv$, and
  \begin{equation}
    \label{eq:L2est} \| \gam \|_{L^2  (\Om, \RR^3)} \leqslant C \| \gv \|_{L^2
    (\Om, \RR^3)} .
  \end{equation}
  If, in addition, $\partial \Om$ is of class $C^{1, 1}$ or $\Om$ is convex,
  then $\gam \in H^1 (\Om, \RR^3)$ and
  \begin{equation}
    \label{eq:H1est} \| \gam \|_{H^1 (\Om, \RR^3)} \leqslant C \| \gv \|_{L^2 
    (\Om, \RR^3)} .
  \end{equation}
\end{theorem}

\begin{remark}[on Friedrichs--Gaffney inequality]
  \label{rem:friedrichs} The improved estimate \eqref{eq:H1est} is the
  Friedrichs--Gaffney (or Maxwell) inequality. It arises by combining the
  gauge conditions $\divg \gam = 0$ and $\gam \cdot \nuvec = 0$ with the
  vector identity $- \Delta \gam = \curl \curl \gam - \nabla \divg \gam$,
  which reduces to $- \Delta \gam = \curl \gv$; classical elliptic regularity
  then yields $\gam \in H^1$ under either $C^{1, 1}$ boundary regularity or
  convexity \cite[Thm.~I.2.12]{Girault2012}. On non-convex Lipschitz domains
  the inclusion $\XN \subset H^1 (\Om, \RR^3)$ may fail: at reentrant corners,
  the vector potential typically exhibits a corner singularity that prevents
  $\nabla \gam \in L^2$ even when $\gv \in L^2$. The geometric hypothesis in
  \eqref{eq:thm1H1} is therefore essential.
\end{remark}

\subsection{Proof of Theorem~\ref{thm:thm1}}\label{ssec:proofthm1}We divide
the proof into two main steps

\begin{proof}[Proof of Theorem~\ref{thm:thm1}]

\smallskip
\noindent  \emph{Step 1. Reformulation in terms of divergence-free fields.} We
  work with the dual functional $\mathcal{A}_{\eps} (\nn, \gv)$ introduced in
  \eqref{eq:Adual}, restricted to divergence-free fields $\gv \in \Vspace$.
  Define the candidate minimizer
  \begin{equation}
    \label{eq:minimizerg} \gv_{\ast} := - A \nabla \unn,
  \end{equation}
  where $\unn = \vnn + \xis$ is the solution of \eqref{eq:Poisson}. By
  \eqref{eq:Poisson}, $\divg (A \nabla \unn) = 0$, so $\gv_{\ast} \in
  \Vspace$. We claim that
  \begin{equation}
    \label{eq:claim1} \min_{\gv \in \Vspace} \mathcal{A}_{\eps} (\nn, \gv)
    \hspace{0.27em} = \hspace{0.27em} \mathcal{A}_{\eps} (\nn, \gv_{\ast})
    \hspace{0.27em} = \hspace{0.27em} \mathcal{W}_E (\nn, \vnn),
  \end{equation}
  and that $\gv_{\ast}$ is the unique minimizer.
  
  The argument is a quadratic completion based on the algebraic identity
  \begin{equation}
    \label{eq:keyidentity} \gv + A \nabla \xis \hspace{0.27em} =
    \hspace{0.27em} (\gv - \gv_{\ast}) - A \nabla \vnn,
  \end{equation}
  valid because $A \nabla \xis = A \nabla \unn - A \nabla \vnn = - \gv_{\ast}
  - A \nabla \vnn$. Expanding the quadratic form $A^{- 1} (\bullet) \cdot
  (\bullet)$ applied to \eqref{eq:keyidentity} and integrating over $\Om$, we
  get
  \begin{equation}
    \label{eq:Aexpand} \mathcal{A}_{\eps} (\nn, \gv) = \int_{\Om} A^{- 1} 
    (\gv - \gv_{\ast}) \cdot (\gv - \gv_{\ast}) \dx - 2 \hspace{-0.17em}
    \int_{\Om} (\gv - \gv_{\ast}) \cdot \nabla \vnn \dx + \int_{\Om} A \nabla
    \vnn \cdot \nabla \vnn \dx,
  \end{equation}
  where for the cross term we used $A^{- 1} A = I$.
  
  Observe that the cross integral in \eqref{eq:Aexpand} vanishes. Indeed, $\gv -
  \gv_{\ast} \in \Vspace$ is divergence-free in the distributional sense, and
  $\vnn \in H^1_0 (\Om)$ vanishes on $\partial \Om$. The duality pairing in
  \eqref{eq:normaltrace} applied to $\gv - \gv_{\ast} \in \Hdiv$ and $\vnn \in
  H^1 (\Om)$ gives
  \begin{equation}
    \label{eq:crossvanish} \int_{\Om} (\gv - \gv_{\ast}) \cdot \nabla \vnn \dx
    = - \int_{\Om} \divg (\gv - \gv_{\ast}) \hspace{0.17em} \vnn \dx + \langle
    (\gv - \gv_{\ast}) \cdot \nuvec, \hspace{0.17em} \vnn \rangle_{\partial
    \Om} = 0,
  \end{equation}
  since $\divg (\gv - \gv_{\ast}) = 0$ in $\Om$ and the trace of $\vnn$
  vanishes in $H^{1 / 2} (\partial \Om)$.
  
  Also, the last term in \eqref{eq:Aexpand} identifies with $\mathcal{W}_E
  (\nn, \vnn)$. Indeed, testing \eqref{eq:vsolve} with $\vnn \in H^1_0 (\Om)$
  gives the orthogonality relation
  \begin{equation}
    \label{eq:orth} \int_{\Om} A \nabla \vnn \cdot \nabla \vnn \dx
    \hspace{0.27em} = \hspace{0.27em} - \int_{\Om} A \nabla \xis \cdot \nabla
    \vnn \dx,
  \end{equation}
  whence, from the definition \eqref{eq:WE},
  \begin{equation}
    \label{eq:WEident} \mathcal{W}_E (\nn, \vnn) \hspace{0.27em} =
    \hspace{0.27em} - \hspace{-0.17em} \int_{\Om} A \nabla \vnn \cdot \nabla
    \vnn \dx - 2 \hspace{-0.17em} \int_{\Om} A \nabla \vnn \cdot \nabla \xis
    \dx \hspace{0.27em} = \hspace{0.27em} \int_{\Om} A \nabla \vnn \cdot
    \nabla \vnn \dx .
  \end{equation}
  Combining \eqref{eq:Aexpand}, \eqref{eq:crossvanish}, and
  \eqref{eq:WEident}, we get
  \begin{equation}
    \label{eq:completedsquare} \mathcal{A}_{\eps} (\nn, \gv) \hspace{0.27em} =
    \hspace{0.27em} \int_{\Om} A^{- 1}  (\gv - \gv_{\ast}) \cdot (\gv -
    \gv_{\ast}) \dx \hspace{0.27em} + \hspace{0.27em} \mathcal{W}_E (\nn,
    \vnn) \hspace{0.27em} \geqslant \hspace{0.27em} \mathcal{W}_E (\nn, \vnn),
  \end{equation}
  with equality if and only if $\gv = \gv_{\ast}$ almost everywhere, by the
  uniform positivity of $A^{- 1}$. Since $\gv_{\ast} \in \Vspace$, the minimum
  is attained and unique, establishing \eqref{eq:claim1}.
  
 \smallskip
\noindent  \emph{Step 2. From divergence-free fields to vector potentials.}
  We now relate $\mathcal{A}_{\eps} (\nn, \cdot)$ on $\Vspace$ to its lift on the
  various spaces of vector potentials. By construction, for $\gam \in \Hcurl$,
  \begin{equation} \label{eq:Atocirc}
  \mathcal{A}_{\eps} (\nn, \curl \gam) = \min_{\gv \in
    \Vspace} \mathcal{A}_{\eps} (\nn, \gv) \hspace{0.27em} \Longleftrightarrow
    \hspace{0.27em} \curl \gam = \gv_{\ast} .
  \end{equation}
  The desired identities then follow from a careful description of the range
  of the curl operator on each domain.
  \begin{enumerate}[(i)]
    \item \emph{Minimization over} $\XN$. By Theorem~\ref{thm:isoCurl}, the
    operator $\curl : \XN \to \Vspace$ is a bijection. Hence for every $\gv
    \in \Vspace$ there is a unique $\gam \in \XN$ with $\curl \gam = \gv$, and
    minimization of $\mathcal{A}_{\eps} (\nn, \curl \gam)$ over $\gam \in \XN$
    is equivalent to minimization of $\mathcal{A}_{\eps} (\nn, \gv)$ over $\gv
    \in \Vspace$. The unique minimum is attained at the unique element
    $\gam_{\ast} \in \XN$ satisfying $\curl \gam_{\ast} = \gv_{\ast} = - A
    \nabla \unn$, which is precisely \eqref{eq:curlgamma}.
    \smallskip
    
    \item \emph{Minimization over} $\Hcurl$. Trivially $\XN \subset
    \Hcurl$, so
    \[ \min_{\gam \in \Hcurl} \mathcal{A}_{\eps} (\nn, \curl  \gam) \leqslant
       \min_{\gam \in \XN} \mathcal{A}_{\eps} (\nn, \curl  \gam)
       =\mathcal{A}_{\eps} (\nn, \curl \gam_{\ast}) . \]
    
    Conversely, for any $\gam \in \Hcurl$, $\curl \gam \in \Vspace$, so
    \[ \mathcal{A}_{\eps} (\nn, \curl \gam) \geqslant
       \min_{\bm{g} \in \Vspace} \mathcal{A}_{\eps} (\nn,
       \bm{g}) =\mathcal{A}_{\eps} (\nn, \curl
       \gam_{\ast}) . \]
    The two minima coincide; however, uniqueness fails on $\Hcurl$ because of
    the invariance of the $\curl$ operator on gradient fields.
    \smallskip
    
    \item \emph{Minimization over $H^1 (\Om, \RR^3)$ under $C^{1, 1}$
    regularity.} If $\partial \Om$ is of class $C^{1, 1}$ (or $\Om$ is
    convex), Theorem~\ref{thm:isoCurl} yields $\gam_{\ast} \in H^1 (\Om,
    \RR^3)$, so the unique $\XN$-minimizer already lies in $H^1$. Hence
    \[ \min_{\gam \in H^1 (\Om, \RR^3) \cap \XN} \mathcal{A}_{\eps} (\nn,
       \curl \gam) \hspace{0.27em} = \hspace{0.27em} \mathcal{A}_{\eps} (\nn,
       \curl \gam_{\ast}) \hspace{0.27em} = \hspace{0.27em} \min_{\gam \in
       \XN} \mathcal{A}_{\eps} (\nn, \curl \gam) . \]
    For the unconstrained minimization over $H^1 (\Om, \RR^3)$, the $\curl$
    operator is surjective onto $\Vspace$ under the same regularity (any
    $\XN$-representative may be modified by a curl-free $H^1$ gradient, which
    lies in $H^1$ under the present regularity), so the minimum value
    coincides with $\min_{\bm{g} \in \Vspace}
    \mathcal{A}_{\eps} (\nn, \bm{g})$.
  \end{enumerate}
  
  This completes the proof of \eqref{eq:thm1main} and \eqref{eq:thm1H1}.
\end{proof}

\begin{remark}[Heuristic origin of $\mathcal{A}_{\eps}$ via Fenchel--Rockafellar]
  \label{rem:FR}The dual functional $\mathcal{A}_{\eps}$
  can also be motivated through Fenchel--Rockafellar duality
  \cite{Temam1999}. Writing $\mathcal{W}_E (\nn, \cdot) = - \varphi - \psi$
  on $H^1_0 (\Om)$ with
  \[ \varphi (v) := \int_{\Om} A \nabla v \cdot \nabla v \dx, \qquad \psi (v)
     := 2 \hspace{-0.17em} \int_{\Om} A \nabla \xis \cdot \nabla v \dx, \]
  the Fenchel--Rockafellar theorem identifies
  \begin{equation}
    \label{eq:FR} \max_{v \in H^1_0 (\Om)} \mathcal{W}_E (\nn, v)
    \hspace{0.27em} = \hspace{0.27em} \min_{f \in H^{- 1} (\Om)}
    [\varphi^{\ast} (- f) + \psi^{\ast} (f)] .
  \end{equation}
  Identifying $H^{- 1} (\Om)$ with $- \divg L^2  (\Om, \RR^3)$ and
  parametrizing $f = - \divg \bm{h}$ with
  $\bm{h} \in L^2  (\Om, \RR^3)$, the $L^2$ Hodge
  decomposition splits $\bm{h}$ into a divergence-free
  part $\gv \in \Vspace$ and a gradient. After computing the conjugates
  $\varphi^{\ast}$ and $\psi^{\ast}$ and eliminating the gradient component,
  the minimization reduces to one over $\Vspace$ with integrand given
  precisely $\mathcal{A}_{\eps} (\nn, \gv)$ as in \eqref{eq:Adual}.
  Theorem~\ref{thm:isoCurl} subsequently identifies $\Vspace$ as $\curl
  (\XN)$, yielding the vector-potential representation $\mathcal{A}_{\eps}
  (\nn, \curl \gam)$. This constructive derivation is conceptually
  illuminating and, indeed, was the route through which we originally obtained
  the explicit form of the dual functional. However, it is rather laborious.
  Once the dual functional has been identified, its correctness can instead be
  verified directly through the quadratic-completion argument above, which is
  both shorter and more transparent. For this reason, we adopt the latter
  approach in the presentation.
\end{remark}

\section{Asymptotic limit and quantitative rates: proof of
Theorem~\ref{thm:thm2}}\label{sec:asymptotic}

The proof unfolds in three parts: an algebraic comparison yielding the energy
bound \eqref{eq:energybound}; a $\Gamma$-convergence argument supplying the
strong $H^1$-convergence of minimizers; and a perturbative coercivity transfer
that converts the energy gap into the quantitative rate \eqref{eq:rate}.

\subsection{The energy gap: proof of \texorpdfstring{\eqref{eq:energybound}}{the error bound}} \label{ssec:energygap}By the decomposition
\eqref{eq:Feps}, the gap between $\mathcal{F}_{\varepsilon}$ and
$\mathcal{G}_{\varepsilon}$ is precisely the nonlocal electrostatic
functional:
\begin{equation}
  \label{eq:gapisA} \mathcal{F}_{\varepsilon} (\nn, \gv)
  -\mathcal{G}_{\varepsilon} (\nn) \hspace{0.27em} = \hspace{0.27em}
  \mathcal{A}_{\varepsilon} (\nn, \gv) \hspace{0.27em} \geqslant
  \hspace{0.27em} 0 \qquad \forall \hspace{0.17em} \gv \in \Vspace,
\end{equation}
the nonnegativity following from the positive-definiteness of $\Aeps (\nn)^{-
1}$. The bound \eqref{eq:energybound} therefore reduces to a uniform upper
bound on $\min_{\gv} \mathcal{A}_{\varepsilon} (\nn, \cdot)$, which is the
object of the following Lemma. Indeed, once Lemma~\ref{lem:energygap} is
proved, estimate \eqref{eq:energybound} of Theorem~\ref{thm:thm2} follows
immediately since the constant $C_{\xi}$ in \eqref{eq:gapuniform} does not
depend on $\nn$.

\begin{lemma}[Energy gap]
  \label{lem:energygap}There exists a constant $C_{\xi} > 0$, depending only
  on $\| \nabla \xis \|_{L^{\infty}}$ and $| \Om |$, such that for every
  $\varepsilon \in (0, 1)$ and every $\nn \in H^1 (\Om, \jM)$,
  \begin{equation}
    \label{eq:gapuniform} 0 \hspace{0.27em} \leqslant \hspace{0.27em}
    \min_{\gv \in \Vspace} \mathcal{A}_{\varepsilon} (\nn, \gv)
    \hspace{0.27em} = \hspace{0.27em} \min_{\gv \in \Vspace}
    \mathcal{F}_{\varepsilon} (\nn, \gv) -\mathcal{G}_{\varepsilon} (\nn)
    \hspace{0.27em} \leqslant \hspace{0.27em} C_{\xi}  \hspace{0.17em}
    \varepsilon^2 .
  \end{equation}
  Moreover, the unique minimizer $\gv_{\ast} \in \Vspace$ satisfies
  \begin{equation}
    \label{eq:gminbound} \| \gv_{\ast} + \nabla \xis \|_{L^2 (\Om)}
    \hspace{0.27em} \leqslant \hspace{0.27em} \sqrt{C_{\xi}} \varepsilon .
  \end{equation}
\end{lemma}

\begin{remark}[Physical interpretation]
  The unique minimizer $\gv_{\ast} = - \Aeps (\nn) \nabla \unn$ represents the
  dimensionless electric displacement field. Its membership in the
  divergence-free space $\Vspace$ precisely encodes Gauss's law for
  dielectrics in the absence of free charges. Furthermore, the estimate
  \eqref{eq:gminbound} provides a rigorous bound on the dielectric
  back-reaction: it guarantees that the overall $L^2$-deviation of the actual
  displacement field $\gv_{\ast}$ from the rigid, isotropic background field
  $- \nabla \xis$ is strictly controlled, scaling at most \emph{linearly}
  with the dielectric anisotropy $\varepsilon$.
\end{remark}

\begin{remark}
  For the proof we use the algebraic identity (see Sherman--Morrison formula
  for the inverse of a rank-1 update)
  \begin{equation}
    \label{eq:idmatrix} (I + \varepsilon \hspace{0.17em} \nn \otimes \nn)^{-
    1} = I - \frac{\varepsilon}{1 + \varepsilon} \hspace{0.17em} \nn \otimes
    \nn,
  \end{equation}
  valid for $\varepsilon > - 1$ and any $\nn \in \jM$.
\end{remark}

\begin{proof}
  Throughout the proof, $C_{\xi}$ denotes a generic positive constant
  depending only on $\| \nabla \xis \|_{L^{\infty}}$ and $| \Om |$, which may
  change from line to line. Nonnegativity in \eqref{eq:gapuniform} is
  \eqref{eq:gapisA}; we need only the upper bound.
  
  \smallskip
  
\noindent \emph{Upper bound.} By harmonicity of $\xis$, $- \nabla \xis
  \in \Vspace$, and it may be used as a competitor. With $\gv = - \nabla
  \xis$, the inner integrand of $\mathcal{A}_{\varepsilon}$ becomes
  \begin{align}
    \gv + (I + \varepsilon \hspace{0.17em} \nn \otimes \nn) \nabla \xis =
    \varepsilon (\nn \otimes \nn) \nabla \xis = \varepsilon (\nn \cdot \nabla
    \xis) \hspace{0.17em} \nn . & 
  \end{align}
  Hence
  \begin{align}
    \mathcal{A}_{\varepsilon} (\nn, - \nabla \xis) & = \int_{\Om} \left[
    \left( I - \tfrac{\varepsilon}{1 + \varepsilon} \nn \otimes \nn \right)
    \varepsilon (\nn \cdot \nabla \xis) \hspace{0.17em} \nn \right] \cdot
    \varepsilon (\nn \cdot \nabla \xis) \hspace{0.17em} \nn  \dx \nonumber\\
    & = \frac{\varepsilon^2}{1 + \varepsilon}  \hspace{-0.17em} \int_{\Om}
    (\nn \cdot \nabla \xis)^2 \dx, 
  \end{align}
  and therefore,
  \begin{equation}
    \label{eq:upperbound} \min_{\gv \in \Vspace} \mathcal{A}_{\varepsilon}
    (\nn, \gv) \hspace{0.27em} \leqslant \hspace{0.27em}
    \mathcal{A}_{\varepsilon} (\nn, - \nabla \xis) \hspace{0.27em} =
    \hspace{0.27em} \frac{\varepsilon^2}{1 + \varepsilon}  \hspace{-0.17em}
    \int_{\Om} (\nn \cdot \nabla \xis)^2 \dx \hspace{0.27em} \leqslant
    \hspace{0.27em} C_{\xi}  \hspace{0.17em} \varepsilon^2,
  \end{equation}
  with $C_{\xi} := 4 | \Om | \| \nabla \xis \|_{L^{\infty}}^2$ and the factor
  $4$ used here so that the same $C_{\xi}$ can be used also for the refined
  lower bound below.
  
   \smallskip
  
\noindent\emph{Refined lower bound.} Let $\gv \in \Vspace$. To simplify notation,
  we introduce the vector field $\bm{w} := \gv + \Aeps
  (\nn) \nabla \xis$, which expands to $\bm{w} = (\gv +
  \nabla \xis) + \varepsilon (\nn \cdot \nabla \xis) \nn$.
  Using the algebraic identity \eqref{eq:idmatrix} for the inverse matrix
  $\Aeps (\nn)^{- 1}$, the integrand of the dual functional can be pointwise
  bounded from below:
  \begin{equation}
    \Aeps (\nn)^{- 1} \bm{w} \cdot
    \bm{w} \hspace{0.27em} = \hspace{0.27em} \left[
    \left( I - \frac{\varepsilon}{1 + \varepsilon}  \nn \otimes \nn \right)
    \bm{w} \right] \cdot \bm{w}
    \hspace{0.27em} \geqslant \hspace{0.27em} \frac{1}{1 + \varepsilon}
    \hspace{0.17em} |\bm{w}|^2 .
  \end{equation}
  We then apply the elementary algebraic Young's inequality $|a + b|^2
  \geqslant \frac{1}{2} |a|^2 - |b|^2$ to the expanded form of
  $\bm{w}$. This yields:
  \begin{equation}
    |\bm{w}|^2 \hspace{0.27em} \geqslant \hspace{0.27em}
    \frac{1}{2} | \gv + \nabla \xis |^2 - \varepsilon^2  (\nn \cdot \nabla
    \xis)^2 .
  \end{equation}
  Substituting this into the pointwise bound and integrating over the domain
  $\Om$, we obtain the refined lower bound for the functional:
  \begin{align}
    \mathcal{A}_{\varepsilon} (\nn, \gv) & \geqslant \frac{1}{2 (1 +
    \varepsilon)} \hspace{0.17em} \| \gv + \nabla \xis \|_{L^2}^2 -
    \frac{\varepsilon^2}{1 + \varepsilon}  \hspace{-0.17em} \int_{\Om} (\nn
    \cdot \nabla \xis)^2 \dx .  \label{eq:lowerbound}
  \end{align}
  Note that \eqref{eq:gapuniform} is already established by
  \eqref{eq:upperbound} and the trivial bound $\min \mathcal{A}_{\varepsilon}
  \geqslant 0$. We now use \eqref{eq:lowerbound} solely to establish the
  $L^2$-estimate \eqref{eq:gminbound}. Rearranging \eqref{eq:lowerbound}
  evaluated at the minimizer $\gv_{\ast}$, and applying the upper bound
  \eqref{eq:upperbound} to $\mathcal{A}_{\varepsilon} (\nn, \gv_{\ast})$, we
  get:
  \begin{align}
    \frac{1}{2 (1 + \varepsilon)}  \hspace{0.17em} \| \gv_{\ast} + \nabla \xis
    \|_{L^2}^2 & \hspace{0.27em} \leqslant \hspace{0.27em}
    \mathcal{A}_{\varepsilon} (\nn, \gv_{\ast}) + \frac{\varepsilon^2}{1 +
    \varepsilon}  \int_{\Om} (\nn \cdot \nabla \xis)^2 \dx \\
    & \hspace{0.27em} \leqslant \hspace{0.27em} \frac{2 \hspace{0.17em}
    \varepsilon^2}{1 + \varepsilon}  \int_{\Om} (\nn \cdot \nabla \xis)^2 \dx. 
  \end{align}
  Multiplying both sides by $2 (1 + \varepsilon)$ gives:
  \begin{equation}
    \| \gv_{\ast} + \nabla \xis \|_{L^2}^2 \hspace{0.27em} \leqslant
    \hspace{0.27em} 4 \varepsilon^2  \int_{\Om} (\nn \cdot \nabla \xis)^2 \dx
    \hspace{0.27em} \leqslant \hspace{0.27em} C_{\xi}  \hspace{0.17em}
    \varepsilon^2,
  \end{equation}
  where the spatial factor $4 | \Om | \| \nabla \xis \|_{L^{\infty}}^2$ has
  been absorbed into the generic constant $C_{\xi}$.
\end{proof}

\subsection{\texorpdfstring{$\Gamma$}{Gamma}-convergence and \texorpdfstring{$H^1$}{H1}-convergence of
minimizers}\label{ssec:gamma}

It is standard to show that the uncoupled family $\mathcal{G}_{\varepsilon}$
$\Gamma$-converges to $\mathcal{G}_0$.

\begin{proposition}[$\Gamma$-convergence and energy convergence]
  \label{prop:gamma}The uncoupled family $\mathcal{G}_{\varepsilon}$
  $\Gamma$-converges to $\mathcal{G}_0$ {{\em (\/}}see {{\em
  \eqref{eq:G0fun}\/}}{{\em )\/}} in the weak topology of $H^1 (\Om)$ as
  $\varepsilon \to 0$. Moreover, for any family $\{ \nse \}_{\varepsilon > 0}$
  of global minimizers of $\mathcal{G}_{\varepsilon}$ in $\mathcal{A}_{\nn_b}
  (\Om)$, there exists a subsequence (not relabeled) and a global minimizer
  $\ns$ of $\mathcal{G}_0$ in $\mathcal{A}_{\nn_b} (\Om)$ such that
  \begin{equation}
    \label{eq:strongH1conv} \nse \to \ns \quad \text{strongly in } H^1 (\Om,
    \RR^3)  \text{as } \varepsilon \to 0.
  \end{equation}
\end{proposition}

\begin{proof}
  For the $\Gamma$-convergence in the weak topology of $H^1 (\Om)$ it is
  sufficient to recall that for any $\nn \in \mathcal{A}_{\nn_b} (\Om)$, the
  functional $\mathcal{G}_{\varepsilon} (\nn)$ relates to the limit functional
  $\mathcal{G}_0 (\nn)$ via the identity
  \begin{equation}
    \label{eq:Geps_G0_relation} \mathcal{G}_{\varepsilon} (\nn)
    \hspace{0.27em} = \hspace{0.27em} \mathcal{G}_0 (\nn) - \varepsilon
    \int_{\Om} (\nn \cdot \nabla \xis)^2 \dx .
  \end{equation}
  Hence, since $| \nn | = 1$ a.e. in $\Om$, the perturbation term is uniformly
  bounded: $0 \leqslant \int_{\Om} (\nn \cdot \nabla \xis)^2 \dx \leqslant \|
  \nabla \xis \|_{L^2 (\Om)}^2$.{\smallskip}
  
  {\noindent}{{\em Convergence of minimizers.\/}} Let $\{ \nse \}_{\varepsilon
  > 0}$ be a family of global minimizers of $\mathcal{G}_{\varepsilon}$. By
  testing $\mathcal{G}_{\varepsilon}$ against the fixed boundary extension
  $\nn_b \in \mathcal{A}_{\nn_b} (\Om)$, we see that
  $\mathcal{G}_{\varepsilon} (\nse) \leqslant \mathcal{G}_{\varepsilon}
  (\nn_b) \leqslant \mathcal{G}_0 (\nn_b) < \infty$. Also, by
  \eqref{eq:Geps_G0_relation} it is clear that $\{ \nse \}_{\varepsilon > 0}$
  is uniformly bounded in $H^1$; hence, there exists a subsequence such that
  $\nse \rightharpoonup \ns$ weakly in $H^1 (\Om, \RR^3)$ and strongly in $L^2
  (\Om, \RR^3)$, with $\ns \in \mathcal{A}_{\nn_b} (\Om)$. By the fundamental
  theorem of $\Gamma$-convergence, $\ns$ is a global minimizer of
  $\mathcal{G}_0$ over $\mathcal{A}_{\nn_b} (\Om)$, and the minimum energies
  converge:
  \begin{equation}
    \label{eq:energy_conv} \lim_{\varepsilon \to 0} \mathcal{G}_{\varepsilon}
    (\nse) \hspace{0.27em} = \hspace{0.27em} \mathcal{G}_0 (\ns) .
  \end{equation}
  To conclude the strong $H^1$-convergence, we isolate the Dirichlet energy of
  $\nse$ using \eqref{eq:Geps_G0_relation}:
  \begin{equation}
    \alpha \| \nabla \nse \|_{L^2 (\Om)}^2 \hspace{0.27em} = \hspace{0.27em}
    \mathcal{G}_{\varepsilon} (\nse) + \int_{\Om} | \nabla \xis |^2 \dx +
    \varepsilon \int_{\Om} (\nse \cdot \nabla \xis)^2 \dx .
  \end{equation}
  Passing to the limit as $\varepsilon \to 0$ and invoking
  \eqref{eq:energy_conv} together with the uniform bound on the perturbation
  integral, we deduce
  \begin{equation}
    \lim_{\varepsilon \to 0} \alpha \| \nabla \nse \|_{L^2 (\Om)}^2
    \hspace{0.27em} = \hspace{0.27em} \mathcal{G}_0 (\ns) + \int_{\Om} |
    \nabla \xis |^2 \dx \hspace{0.27em} = \hspace{0.27em} \alpha \| \nabla \ns
    \|_{L^2 (\Om)}^2 .
  \end{equation}
  Because $\{ \nse \}_{\varepsilon > 0}$ converges to $\ns$ weakly in $H^1
  (\Om, \RR^3)$ and their gradient $L^2$-norms converge to the gradient
  $L^2$-norm of the weak limit, $\{\nabla \nse \}_{\varepsilon > 0}$ converges
  strongly in $L^2 (\Om, \RR^3)$. Overall, $\nse \to \ns$ strongly in $H^1
  (\Om, \RR^3)$, completing the proof.
\end{proof}

As a byproduct of Proposition~\ref{prop:gamma} and the energy estimate \eqref{eq:energybound} already established in Theorem~\ref{thm:thm2}, it follows that the coupled family $\min_{\gv \in \Vspace} \mathcal{F}_{\varepsilon} (\nn,
\gv)$ also $\Gamma$-converges to $\mathcal{G}_0$.

\begin{proposition}[$\Gamma$-convergence of the true physical energy]\label{prop:gamma2}
  The coupled energy
  \begin{equation}
    \mathcal{E}_{\varepsilon} (\nn) = \min_{\gv \in \Vspace}
    \mathcal{F}_{\varepsilon} (\nn, \gv) .
  \end{equation}
  $\Gamma$-converges to $\mathcal{G}_0$ in the weak topology of $H^1 (\Om)$ as
  $\varepsilon \to 0$.
  
  Furthermore, for any family $\{ \nse \}_{\varepsilon > 0}$ of global
  minimizers of $\mathcal{E}_{\varepsilon}$ in $\mathcal{A}_{\nn_b} (\Om)$,
  there exists a subsequence (not relabeled) and a global minimizer $\ns$ of
  $\mathcal{G}_0$ in $\mathcal{A}_{\nn_b} (\Om)$ such that
  \begin{equation}
    \label{eq:strongH1conv2} \nse \to \ns \quad \text{strongly in } H^1 (\Om,
    \RR^3)  \text{as } \varepsilon \to 0.
  \end{equation}
\end{proposition}

\begin{proof}
  {{\em Step 1. Uniform energetic bounds.\/}} From the explicit definitions in
  \eqref{eq:Geps} and \eqref{eq:G0fun}, the difference between the local
  approximation and the unperturbed base energy is:
  \begin{equation}
    \mathcal{G}_{\varepsilon} (\nn) -\mathcal{G}_0 (\nn) = \int_{\Om} \left( I
    - \Aeps (\nn) \right) \nabla \xis \cdot \nabla \xis \dx .
  \end{equation}
  Under the standard approximation of the dielectric tensor $\Aeps (\nn) = I +
  \varepsilon \nn \otimes \nn$, this difference evaluates exactly to $-
  \varepsilon \int_{\Om} (\nn \cdot \nabla \xis)^2 \dx$. Because admissible
  fields $\nn \in \mathcal{A}_{\nn_b} (\Om)$ satisfy the unit-length
  constraint $| \nn | = 1$ almost everywhere, we obtain the global linear
  bound $|\mathcal{G}_{\varepsilon} (\nn) -\mathcal{G}_0 (\nn) | \leqslant
  \varepsilon \| \nabla \xis \|_{L^2 (\Om)}^2$.
  
  Combining this with the quantitative bound from Theorem~\ref{thm:thm2},
  namely $\sup_{\nn} |\mathcal{E}_{\varepsilon} (\nn)
  -\mathcal{G}_{\varepsilon} (\nn) | \leqslant C_{\xi} \varepsilon^2$, the
  triangle inequality yields a uniform proximity bound between the true
  physical energy and the limit energy:
  \begin{equation}
    \label{eq:uniform_E_G0} \sup_{\nn \in \mathcal{A}_{\nn_b} (\Om)}
    |\mathcal{E}_{\varepsilon} (\nn) -\mathcal{G}_0 (\nn) | \hspace{0.27em}
    \leqslant \hspace{0.27em} C' \varepsilon,
  \end{equation}
  for a constant $C' > 0$ independent of $\nn$.
  \smallskip
  
  \noindent{{\em Step 2. Weak $\Gamma$-convergence $\mathcal{E}_{\varepsilon}
  \xrightarrow{\Gamma} \mathcal{G}_0$.\/}} We first establish the
  $\Gamma$-liminf inequality. Let $\{ \nn_{\varepsilon} \} \subset
  \mathcal{A}_{\nn_b} (\Om)$ be an arbitrary sequence such that
  $\nn_{\varepsilon} \rightharpoonup \nn$ weakly in $H^1 (\Om)$. Using
  \eqref{eq:uniform_E_G0}, we isolate the limit energy from below:
  \begin{equation}
    \mathcal{E}_{\varepsilon} (\nn_{\varepsilon}) \geqslant \mathcal{G}_0
    (\nn_{\varepsilon}) - C' \varepsilon .
  \end{equation}
  Taking the limit inferior as $\varepsilon \to 0$, and invoking the weak
  lower semicontinuity of the Dirichlet energy on $H^1 (\Om)$, we deduce:
  \begin{equation}
    \liminf_{\varepsilon \to 0} \mathcal{E}_{\varepsilon} (\nn_{\varepsilon})
    \geqslant \liminf_{\varepsilon \to 0} \mathcal{G}_0 (\nn_{\varepsilon})
    \geqslant \mathcal{G}_0 (\nn) .
  \end{equation}
  For the $\Gamma$-limsup inequality, let $\nn \in \mathcal{A}_{\nn_b} (\Om)$.
  Selecting the constant recovery sequence $\widetilde{\nn}_{\varepsilon}
  \equiv \nn$ ensures strong (and thus weak) convergence to $\nn$ in $H^1
  (\Om)$. Applying the uniform bound \eqref{eq:uniform_E_G0} from above:
  \begin{equation}
    \limsup_{\varepsilon \to 0} \mathcal{E}_{\varepsilon}
    (\widetilde{\nn}_{\varepsilon}) \leqslant \limsup_{\varepsilon \to 0}
    \left( \mathcal{G}_0 (\nn) + C' \varepsilon \right) =\mathcal{G}_0 (\nn) .
  \end{equation}
  Together, these inequalities strictly confirm that
  $\mathcal{E}_{\varepsilon} \xrightarrow{\Gamma} \mathcal{G}_0$ in the weak
  $H^1 (\Om)$ topology.
\smallskip
  
  \noindent
  {{\em Step 3. Strong $H^1$ convergence of minimizers.\/}} Let $\nse \in \arg
  \min_{\mathcal{A}_{\nn_b} (\Om)} \mathcal{E}_{\varepsilon}$ be a sequence of
  global minimizers, and fix a smooth comparison state $\widehat{\nn} \in
  \mathcal{A}_{\nn_b} (\Om)$. By the minimality of $\nse$ and the bound
  $\mathcal{G}_0 (\nse) \leqslant \mathcal{E}_{\varepsilon} (\nse) + C'
  \varepsilon$, we get that the sequence $\{\mathcal{G}_0
  (\nse)\}_{\varepsilon > 0}$ is uniformly bounded. Hence, the sequence $\{
  \nse \}_{\varepsilon > 0}$ is uniformly bounded in $H^1 (\Om)$ and,
  therefore, there exists a subsequence (not relabeled) and a state $\ns \in
  \mathcal{A}_{\nn_b} (\Om)$ such that $\nse \rightharpoonup \ns$ weakly in
  $H^1 (\Om)$.
  
  By the Fundamental Theorem of $\Gamma$-convergence, $\ns \in \arg
  \min_{\mathcal{A}_{\nn_b} (\Om)} \mathcal{G}_0$. Furthermore, the minimal
  energies converge:
  \begin{equation}
    \label{eq:energy_conv2} \lim_{\varepsilon \to 0} \mathcal{E}_{\varepsilon}
    (\nse) =\mathcal{G}_0 (\ns) .
  \end{equation}
  In particular, by \eqref{eq:uniform_E_G0},
  \[ \lim_{\varepsilon \to 0} \mathcal{G}_0 (\nse) = \lim_{\varepsilon \to 0}
     \mathcal{E}_{\varepsilon} (\nse) + \lim_{\varepsilon \rightarrow 0}
     \left( \mathcal{G}_0 (\nse) -\mathcal{E}_{\varepsilon} (\nse) \right)
     =\mathcal{G}_0 (\ns) . \]
  We elevate the weak convergence to strong convergence by exploiting the
  exact quadratic structure of the uncoupled base energy. From
  \eqref{eq:G0fun}, the limit functional is precisely $\mathcal{G}_0 (\nn) =
  \alpha \| \nabla \nn \|_{L^2 (\Om)}^2 - \| \nabla \xis \|_{L^2 (\Om)}^2$.
  Substituting this explicit formulation into the energy convergence
  established in Step 3 yields:
  \begin{equation}
    \lim_{\varepsilon \to 0} \left( \alpha \| \nabla \nse \|_{L^2 (\Om)}^2 -\|
    \nabla \xis \|_{L^2 (\Om)}^2 \right) = \alpha \| \nabla \ns \|_{L^2
    (\Om)}^2 - \| \nabla \xis \|_{L^2 (\Om)}^2,
  \end{equation}
  and, therefore,
  \begin{equation}
    \lim_{\varepsilon \to 0} \| \nabla \nse \|_{L^2 (\Om)} = \| \nabla \ns
    \|_{L^2 (\Om)} .
  \end{equation}
  This shows that $\nabla \nse \to \nabla \ns$ strongly in $L^2$, completing
  the proof.
\end{proof}

\subsection{Coercivity transfer}\label{ssec:coerctransfer}At the limit state
$\ns$, the second variation of the Dirichlet energy is strictly
positive-definite (coercive) by Assumption~\ref{ass:coercivity}. Our goal is
to prove that for sufficiently small $\varepsilon > 0$, the perturbed
minimizers $\nse$ inherit this strict stability. To do this, we must transfer
the coercivity bound to the perturbed state, evaluating it exclusively against
variations $\vv$ that preserve the unit-length physical constraint of the
director (i.e., variations satisfying $| \nse + \vv |^2 = 1$ almost
everywhere). The algebraic mechanics of this transfer rely critically on the
uniform $W^{1, 4}$-convergence $\nse \to \ns$, established in
Appendix~\ref{sec:regularity}, which ensures the gradients of the two states
are sufficiently close.

\begin{lemma}[Coercivity transfer]
  \label{lem:coerctransfer}Suppose that Assumption~\ref{ass:coercivity} holds
  with constant $\alpha_0 > 0$, and that
  \begin{equation}
    \label{eq:uniformreg} \| \nse - \ns \|_{W^{1, 4} (\Om)} \to 0 \quad
    \text{as } \varepsilon \to 0, \qquad \ns \in W^{1, \infty} (\Om) .
  \end{equation}
  Then there exist $\tilde{\delta}, \tilde{\alpha} > 0$ and
  $\varepsilon_{\star} > 0$ such that for all $\varepsilon \in (0,
  \varepsilon_{\star})$ and all $\vv \in H^1_0 (\Om, \RR^3)$ satisfying $| \vv
  + \nse |^2 = 1$ a.e. and $\| \vv \|_{H^1_0 (\Om)} \leqslant \tilde{\delta}$,
  we have
  \begin{equation}
    \label{eq:coerctransfer} \mathbb{I}_{\varepsilon} [\vv] := \int_{\Om} |
    \nabla \vv |^2 \dx - \int_{\Om} | \vv |^2 | \nabla \nse |^2 \dx
    \hspace{0.27em} \geqslant \hspace{0.27em} \tilde{\alpha} \hspace{0.17em}
    \| \vv \|_{H^1_0 (\Om)}^2 .
  \end{equation}
\end{lemma}

\begin{proof}
  Throughout the proof, $C > 0$ denotes a generic constant depending only on
  $\Om$ and $\| \ns \|_{W^{1, \infty} (\Om)}$, which may change from line to
  line. Without loss of generality, we may assume $\alpha_0 \leqslant 1$.
  
  We decompose the variation $\vv = \vpar + \vperp$ pointwise with respect to
  the limit field $\ns$:
  \begin{equation}
    \vpar := (\vv \cdot \ns) \ns, \qquad \vperp := P \vv, \quad \text{where }
    P := I - \ns \otimes \ns .
  \end{equation}
  Note that $\vpar \cdot \vperp = 0$ pointwise, which implies $| \vv |^2 = |
  \vpar |^2 + | \vperp |^2$.
  
   \smallskip
  
\noindent\emph{Step 1. Coercivity on the perpendicular
  part.} By construction, $\vperp \cdot \ns = 0$.
  Assumption~\ref{ass:coercivity} therefore guarantees $\mathcal{L}_0 [\vperp]
  \geqslant \alpha_0 \| \vperp \|_{H^1_0 (\Om)}^2$. We calculate:
  \begin{align}
    \mathbb{I}_{\varepsilon} [\vperp] & =\mathcal{L}_0 [\vperp] + \int_{\Om} |
    \vperp |^2  (| \nabla \ns |^2 - | \nabla \nse |^2) \dx \nonumber\\
    & \geqslant \alpha_0 \| \vperp \|_{H^1_0 (\Om)}^2 - \int_{\Om} | \vperp
    |^2  | \nabla \ns - \nabla \nse |  | \nabla \ns + \nabla \nse | \dx . 
    \label{eq:perpcoerc_inter}
  \end{align}
  Applying H\"older's inequality with exponents $(4, 4, 2)$ to the integral, we
  bound it by $\| \vperp \|_{L^4}^2  \| \nabla \ns - \nabla \nse \|_{L^4}  \|
  \nabla \ns + \nabla \nse \|_{L^4}$. Since $\ns \in W^{1, \infty} (\Om)$ and
  $\nse \to \ns$ in $W^{1, 4} (\Om)$, the sum of the gradients is uniformly
  bounded in $L^4 (\Om)$. By the Sobolev embedding $H^1_0 (\Om)
  \hookrightarrow L^4 (\Om)$ in three dimensions, $\| \vperp \|_{L^4}^2
  \leqslant C \| \vperp \|_{H^1_0 (\Om)}^2$. Consequently,
  \begin{equation}
    \label{eq:perpcoerc} \mathbb{I}_{\varepsilon} [\vperp] \hspace{0.27em}
    \geqslant \hspace{0.27em} \left( \alpha_0 - C\| \nse - \ns \|_{W^{1, 4}
    (\Om)} \right) \| \vperp \|_{H^1_0 (\Om)}^2 \hspace{0.27em} \geqslant
    \hspace{0.27em} \frac{\alpha_0}{2} \| \vperp \|_{H^1_0 (\Om)}^2,
  \end{equation}
  provided $\varepsilon_{\star}$ is chosen small enough.
  
   \smallskip
  
\noindent\emph{Step 2. Algebraic identity for $\vpar$.}
  Expanding the constraint $| \vv + \nse |^2 = 1$ and using $| \nse |^2 = 1$,
  we obtain $| \vv |^2 + 2 \nse \cdot \vv = 0$. Rewriting this as $| \vv |^2 +
  2 \vv \cdot \ns + 2 \vv \cdot (\nse - \ns) = 0$ allows us to isolate the
  projection:
  \begin{equation}
    \label{eq:vparexpansion} \vpar \hspace{0.27em} = \hspace{0.27em} \phi
    \hspace{0.17em} \ns, \qquad \text{where } \phi := \left( \vv \cdot \ns
    \right) = - \frac{1}{2} | \vv |^2 - \vv \cdot (\nse - \ns) .
  \end{equation}
  Taking absolute values, we have the pointwise bound $| \phi | \leqslant
  \frac{1}{2} | \vv |^2 + | \vv |  | \nse - \ns |$. Squaring this, using the
  elementary inequality $(1 / 2) a^2 + a b \leqslant (1 / 2) a^4 + 2 a^2 b^2$,
  and integrating, yields an $L^2$-estimate for
  \begin{align}
    \| \vpar \|_{L^2 (\Om)}^2 & \leqslant C \int_{\Om} \left( \left| \vv
    \right|^4 + | \nse - \ns |^2  \left| \vv \right|^2 \right) \dx \nonumber\\
    & \leqslant C \left( \| \vv \|_{L^4 (\Om)}^2 +\| \nse - \ns
    \|_{L^{\infty} (\Om)}^2 \right) \| \vv \|_{L^4 (\Om)}^2 \nonumber\\
    & \leqslant C \left( \tilde{\delta}^2 +\| \nse - \ns \|_{W^{1, 4}
    (\Om)}^2 \right) \| \vv \|_{H^1_0 (\Om)}^2,  \label{eq:vparL2}
  \end{align}
  where we invoked the continuous embeddings $H^1_0 (\Om) \hookrightarrow L^4
  (\Om)$ and $W^{1, 4} (\Om) \hookrightarrow L^{\infty} (\Om)$ in three
  dimensions, and denoted by $\tilde{\delta}^2$ the quantity appearing in the
  statement---whose value will be specified later---which bounds $\| \vv
  \|_{H^1_0 (\Om)}^2$. 
  
   \smallskip
  
\noindent\emph{Step 3. Recombining the functional
  $\mathbb{I}_{\varepsilon} [\vv]$.} Expanding the full functional
  \eqref{eq:coerctransfer} using $\vv = \vperp + \vpar$, we obtain:
  \begin{equation}
    \mathbb{I}_{\varepsilon} [\vv] =\mathbb{I}_{\varepsilon} [\vperp] + \|
    \vpar \|_{H^1_0 (\Om)}^2 - \int_{\Om} | \vpar |^2 | \nabla \nse |^2 \dx +
    2 \int_{\Om} \nabla \vperp : \nabla \vpar \dx .
  \end{equation}
  Applying the coercivity bound \eqref{eq:perpcoerc} for $\vperp$, we get
  \begin{equation}
    \mathbb{I}_{\varepsilon} [\vv] \geqslant \frac{\alpha_0}{2} \| \vperp
    \|_{H^1_0}^2 + \| \vpar \|_{H^1_0}^2 - \int_{\Om} | \vpar |^2 | \nabla
    \nse |^2 \dx + 2 \int_{\Om} \nabla \vperp : \nabla \vpar \dx .
  \end{equation}
  To reconstruct the full norm $\| \vv \|_{H^1_0}^2 = \| \vperp \|_{H^1_0}^2 +
  \| \vpar \|_{H^1_0}^2 + 2 \int_{\Om} \nabla \vperp : \nabla \vpar \dx$, we
  substitute $\| \vperp \|_{H^1_0}^2 = \| \vv \|_{H^1_0}^2 - \| \vpar
  \|_{H^1_0}^2 - 2 \int_{\Om} \nabla \vperp : \nabla \vpar \dx$ into the
  inequality. This algebraic reorganization yields:
  \begin{equation}
    \mathbb{I}_{\varepsilon} [\vv] \geqslant \frac{\alpha_0}{2} \| \vv
    \|_{H^1_0}^2 + \left( 1 - \frac{\alpha_0}{2} \right) \| \vpar \|_{H^1_0}^2
    + (2 - \alpha_0)  \int_{\Om} \nabla \vperp : \nabla \vpar \dx - \int_{\Om}
    | \vpar |^2 | \nabla \nse |^2 \dx .
  \end{equation}
  Since we assumed $\alpha_0 \leqslant 1$, the coefficient $(1 -
  \frac{\alpha_0}{2})$ is strictly positive. Thus, we may safely drop the
  positive term $(1 - \frac{\alpha_0}{2}) \| \vpar \|_{H^1_0}^2 \geqslant 0$
  to get
  \begin{equation}
    \label{eq:Iexpand} \mathbb{I}_{\varepsilon} [\vv] \geqslant
    \frac{\alpha_0}{2} \| \vv \|_{H^1_0}^2 + (2 - \alpha_0)  \int_{\Om} \nabla
    \vperp : \nabla \vpar \dx - \int_{\Om} | \vpar |^2 | \nabla \nse |^2 \dx .
  \end{equation}
  
  It remains to show that the last two error terms can be absorbed by
  $\frac{\alpha_0}{4} \| \vv \|_{H^1_0}^2$.
  
   \smallskip
  
\noindent\emph{Step 4. Bounding the error terms.} We must control the
  last two terms in \eqref{eq:Iexpand} by a small fraction of $\| \vv
  \|_{H^1_0 (\Om)}^2$. We handle them one by one.
  
  \noindent\textbf{(a) The parallel term.} We rewrite the gradient squared
  by adding and subtracting the limit gradient: $| \nabla \nse |^2 = | \nabla
  \ns |^2 + (| \nabla \nse |^2 - | \nabla \ns |^2)$. Substituting this into
  the integral splits it into two pieces:
  \[ \int_{\Om} | \vpar |^2 | \nabla \nse |^2 \dx \hspace{0.27em} \leqslant
     \hspace{0.27em} \| \nabla \ns \|_{L^{\infty}}^2 \| \vpar \|_{L^2}^2 +
     \int_{\Om} | \vpar |^2  \left| | \nabla \nse |^2 - | \nabla \ns |^2
     \right| \dx . \]
  For the second piece, we factor the difference of squares as $(\nabla \nse -
  \nabla \ns) : (\nabla \nse + \nabla \ns)$, bound $| \vpar | \leqslant | \vv
  |$, and apply H\"older's inequality with exponents $(4, 4, 2)$ exactly as in
  Step 1. Using our $L^2$-bound for $\vpar$ from \eqref{eq:vparL2}, we obtain:
  \begin{align}
    \int_{\Om} | \vpar |^2 | \nabla \nse |^2 \dx & \hspace{0.27em} \leqslant
    \hspace{0.27em} \| \nabla \ns \|_{L^{\infty}}^2  \| \vpar \|_{L^2}^2 + C\|
    \nse - \ns \|_{W^{1, 4}} \| \vv \|_{L^4}^2 \nonumber\\
    & \hspace{0.27em} \leqslant \hspace{0.27em} C \left( \tilde{\delta}^2 +\|
    \nse - \ns \|_{W^{1, 4}}^2 +\| \nse - \ns \|_{W^{1, 4}} \right) \| \vv
    \|_{H^1_0}^2 .  \label{eq:vparmass}
  \end{align}
  \noindent\textbf{(b) The cross-gradient term.} We must bound $\int_{\Om} |
  \nabla \vperp : \nabla \vpar | \dx$. First, we compute the pointwise
  product. Differentiating $\vpar = \phi \ns$ and $\vperp = P \vv$ via the
  product rule yields $\partial_i \vpar = (\partial_i \phi) \ns + \phi
  \partial_i \ns$ and $\partial_i \vperp = P \partial_i \vv + (\partial_i P)
  \vv$. Therefore, taking the inner product of these two expressions gives
  three terms:
  \begin{equation}
    \partial_i \vperp \cdot \partial_i \vpar \hspace{0.27em} = \hspace{0.27em}
    \underbrace{\left( P \partial_i \vv \cdot \ns \right) \partial_i
    \phi}_{\text{Term 1}} \hspace{0.27em} + \hspace{0.27em} \underbrace{\phi
    \left( P \partial_i \vv \cdot \partial_i \ns \right)}_{\text{Term 2}}
    \hspace{0.27em} + \hspace{0.27em} \underbrace{\left( (\partial_i P) \vv
    \cdot \partial_i \vpar \right)}_{\text{Term 3}} .
  \end{equation}
  Because $P$ is the projection onto the orthogonal complement of $\ns$, we
  have $P\bm{w} \cdot \ns = 0$ for any vector
  $\bm{w}$. Thus, the geometric cancellation $P
  \partial_i \vv \cdot \ns = 0$ occurs, and Term 1 identically vanishes.
  
  To bound the remaining terms, we use the properties $| \partial_i P|
  \leqslant 2 | \nabla \ns | \leqslant C$ and $| \nabla \vpar | \leqslant |
  \nabla \phi | + C | \phi |$ for some $C$ depending only on $\ns$. This
  simplifies the pointwise absolute bound to:
  \begin{equation}
    | \nabla \vperp : \nabla \vpar | \hspace{0.27em} \leqslant \hspace{0.27em}
    C \left[ | \phi || \nabla \vv | + | \vv | (| \nabla \phi | + | \phi |)
    \right] .
  \end{equation}
  Next, we need to estimate the gradient of the scalar $\phi$. Recalling from
  \eqref{eq:vparexpansion} that $\phi = - \frac{1}{2} | \vv |^2 - \vv \cdot
  (\nse - \ns)$, the chain rule gives:
  \begin{equation}
    | \nabla \phi | \leqslant \left| \vv \right|  \left| \nabla \vv \right| +
    \left| \nabla \vv \right|  \left| \nse - \ns \right| + \left| \vv \right| 
    \left| \nabla (\nse - \ns) \right|
  \end{equation}
  Substituting $\phi$ and $\nabla \phi$ into the bound above yields a
  polynomial expression consisting of two types of terms: \emph{lower-order
  terms} (purely powers of $\vv$ and $\nabla \vv$ that are bounded by the
  energy norm via Sobolev embeddings) and \emph{perturbation-controlled
  terms} (products involving the difference $(\nse - \ns)$, which vanish as
  $\varepsilon \to 0$ due to the uniform $W^{1, 4}$-convergence established in
  \eqref{eq:uniformreg}).
  \begin{equation}
    | \nabla \vperp : \nabla \vpar | \leqslant C \left( \left| \vv \right|^2 
    \left| \nabla \vv \right| + \left| \vv \right|  \left| \nabla \vv \right| 
    \left| \nse - \ns \right| + \left| \vv \right|^3 + \left| \vv \right|^2
    \left| \nse - \ns \right| + \left| \vv \right|^2 \left| \nabla \left( \nse
    - \ns \right) \right| \right) .
  \end{equation}
  We integrate this inequality over $\Om$. We group the terms into two
  categories and apply standard 3D Sobolev embeddings ($H^1_0 \hookrightarrow
  L^3, L^4, L^6$):
  \begin{itemize}
    \item \emph{Cubic terms in $\vv$:} We bound $\int | \vv |^3 \leqslant C
    \| \vv \|_{H^1_0}^3$ and $\int | \vv |^2 | \nabla \vv | \leqslant \| \vv
    \|_{L^4}^2 \| \nabla \vv \|_{L^2} \leqslant C \| \vv \|_{H^1_0}^3$. Since
    $\| \vv \|_{H^1_0} \leqslant \tilde{\delta}$, we extract a factor of
    $\tilde{\delta}$ to control these by $C \tilde{\delta} \| \vv
    \|_{H^1_0}^2$.
    \smallskip
    
    \item \emph{Perturbation terms:} The terms involving the difference
    $(\nse - \ns)$ and its gradient are bounded using H\"older's inequality
    with appropriate exponents, yielding upper bounds of the form $C \| \nse -
    \ns \|_{W^{1, 4}} \| \vv \|_{H^1_0}^2$.
  \end{itemize}
  Summing these integrated bounds, we arrive at the final estimate for the
  cross-gradient:
  \begin{equation}
    \label{eq:crossgrad} \int_{\Om} | \nabla \vperp : \nabla \vpar | \dx
    \hspace{0.27em} \leqslant \hspace{0.27em} C \left( \tilde{\delta} +\| \nse
    - \ns \|_{W^{1, 4} (\Om)} \right) \| \vv \|_{H^1_0 (\Om)}^2 .
  \end{equation}
  \smallskip
  
\noindent 
  \emph{Conclusion.} Inserting \eqref{eq:vparmass} and
  \eqref{eq:crossgrad} into \eqref{eq:Iexpand}, we obtain:
  \begin{equation}
    \mathbb{I}_{\varepsilon} [\vv] \hspace{0.27em} \geqslant \hspace{0.27em}
    \left[ \frac{\alpha_0}{2} - C \left( \tilde{\delta} + \tilde{\delta}^2 +\|
    \nse - \ns \|_{W^{1, 4}} +\| \nse - \ns \|_{W^{1, 4}}^2 \right) \right] \|
    \vv \|_{H^1_0 (\Om)}^2 .
  \end{equation}
  By first selecting $\tilde{\delta}$ small enough, and then restricting
  $\varepsilon \in (0, \varepsilon_{\star})$ such that $\| \nse - \ns
  \|_{W^{1, 4}} \to 0$ makes the remaining terms sufficiently small, the
  bracketed term is strictly bounded from below by $\frac{\alpha_0}{4}$. This
  establishes \eqref{eq:coerctransfer} with $\tilde{\alpha} :=
  \frac{\alpha_0}{4}$.
\end{proof}

\subsection{Conclusion: proof of Theorem~\ref{thm:thm2}}\label{ssec:proofthm2}

We are now in position to prove Theorem~\ref{thm:thm2}.

\begin{proof}[Proof of Theorem~\ref{thm:thm2}]
  Let $\{ \nse \}_{\varepsilon > 0}$ be a family of
  global minimizers of $\mathcal{G}_{\varepsilon}$ in $\mathcal{A}_{\nn_b}
  (\Om)$. By Proposition~\ref{prop:gamma}, there exist a subsequence (not
  relabeled) and a global minimizer $\ns$ of $\mathcal{G}_0$ with $\nse \to
  \ns$ strongly in $H^1 (\Om, \RR^3)$. This establishes part (i) of
  Theorem~\ref{thm:thm2}.
  
  For part (ii), assume that this specific limit $\ns$ satisfies
  Assumption~\ref{ass:coercivity}. Theorem~\ref{thm:uniformW22est} in
  Appendix~\ref{sec:regularity} strengthens the convergence to $\| \nse - \ns
  \|_{W^{1, 4} (\Om)} \to 0$ and gives $\ns \in W^{1, \infty} (\Om)$;
  hypothesis \eqref{eq:uniformreg} of Lemma~\ref{lem:coerctransfer} is
  therefore satisfied, and the coercivity \eqref{eq:coerctransfer} applies for
  all $\varepsilon$ sufficiently small.
  
  Let $(\bm{m}^{\varepsilon}, \gv_{\ast}^{\varepsilon})
  \in \arg \min \mathcal{F}_{\varepsilon}$ be such that
  $\bm{m}^{\varepsilon} \rightarrow \ns$ in $H^1 (\Om)$,
  and set $\vv_{\varepsilon} := \bm{m}^{\varepsilon} -
  \nse$. Since $\bm{m}^{\varepsilon} = \nn_b = \nse$ on
  $\partial \Om$ in the trace sense, we have $\vv_{\varepsilon} \in H^1_0
  (\Om, \RR^3)$ and $| \nse + \vv_{\varepsilon} | =
  |\bm{m}^{\varepsilon} | = 1$. Moreover, $\|
  \vv_{\varepsilon} \|_{H^1_0} \to 0$ by the triangle inequality combined with
  $\nse \to \ns$ and $\bm{m}^{\varepsilon} \to \ns$.
  
  Choose $\varepsilon_{\bullet} \in (0, \varepsilon_{\star})$ so small that
  $\| \vv_{\varepsilon} \|_{H^1_0} \leqslant \tilde{\delta}$ for $\varepsilon
  \in (0, \varepsilon_{\bullet})$. By Lemma~\ref{lem:coerctransfer},
  \begin{equation}
    \mathbb{I}_{\varepsilon} [\vv_{\varepsilon}] \geqslant \tilde{\alpha}
    \hspace{0.17em} \| \vv_{\varepsilon} \|_{H^1_0}^2 .
  \end{equation}
  We now relate $\mathbb{I}_{\varepsilon} [\vv_{\varepsilon}]$ to the energy
  gap. Since $\nse$ is a critical point of $\mathcal{G}_{\varepsilon}$ on $H^1
  (\Om, \jM)$, it solves the Euler--Lagrange system
  \begin{equation}
    \label{eq:EL} - \alpha \Delta \nse - \varepsilon (\nabla \xis \cdot \nse)
    \nabla \xis = \lambda (\nse) \hspace{0.17em} \nse, \qquad \lambda (\nse) =
    \alpha | \nabla \nse |^2 - \varepsilon (\nabla \xis \cdot \nse)^2,
  \end{equation}
  in the weak sense, with $\nse = \nn_b$ on $\partial \Om$. Testing
  \eqref{eq:EL} with $2 \vv_{\varepsilon} \in H^1_0 (\Om, \RR^3)$, integrating
  by parts, and using the constraint $2 \vv_{\varepsilon} \cdot \nse + |
  \vv_{\varepsilon} |^2 = 0$ to evaluate the Lagrange multiplier contribution
  $\int \lambda (\nse)  \hspace{0.17em} (\nse \cdot 2 \vv_{\varepsilon}) \dx =
  - \int \lambda (\nse) | \vv_{\varepsilon} |^2 \dx$,
  \begin{align}
    2 \alpha \hspace{-0.17em} \int_{\Om} \nabla \nse : \nabla
    \vv_{\varepsilon} \dx & = 2 \varepsilon \hspace{-0.17em} \int_{\Om}
    (\nabla \xis \cdot \nse)  (\nabla \xis \cdot \vv_{\varepsilon}) \dx
    \nonumber\\
    & \quad - \alpha \hspace{-0.17em} \int_{\Om} | \nabla \nse |^2 |
    \vv_{\varepsilon} |^2 \dx + \varepsilon \hspace{-0.17em} \int_{\Om}
    (\nabla \xis \cdot \nse)^2 | \vv_{\varepsilon} |^2 \dx . 
    \label{eq:ELtested}
  \end{align}
  Substituting into the expansion of $\mathcal{G}_{\varepsilon}  (\nse +
  \vv_{\varepsilon}) -\mathcal{G}_{\varepsilon} (\nse)$,
  \begin{align}
    \mathcal{G}_{\varepsilon}  (\nse + \vv_{\varepsilon})
    -\mathcal{G}_{\varepsilon} (\nse) & = \alpha \hspace{-0.17em} \int_{\Om} |
    \nabla \vv_{\varepsilon} |^2 \dx + 2 \alpha \hspace{-0.17em} \int_{\Om}
    \nabla \nse : \nabla \vv_{\varepsilon} \dx \nonumber\\
    & \quad - \varepsilon \hspace{-0.17em} \int_{\Om} (\nabla \xis \cdot
    \vv_{\varepsilon})^2 \dx - 2 \varepsilon \hspace{-0.17em} \int_{\Om}
    (\nabla \xis \cdot \nse)  (\nabla \xis \cdot \vv_{\varepsilon}) \dx
    \nonumber\\
    & = \alpha \hspace{0.17em} \mathbb{I}_{\varepsilon} [\vv_{\varepsilon}] -
    \varepsilon \hspace{-0.17em} \int_{\Om} (\nabla \xis \cdot
    \vv_{\varepsilon})^2 \dx + \varepsilon \hspace{-0.17em} \int_{\Om} (\nabla
    \xis \cdot \nse)^2 | \vv_{\varepsilon} |^2 \dx .  \label{eq:Gexpansion}
  \end{align}
  The last two terms in \eqref{eq:Gexpansion} are bounded in absolute value by
  $C_{\xi} \varepsilon \hspace{0.17em} \| \vv_{\varepsilon} \|_{L^4}^2
  \leqslant C_{\xi} \varepsilon \| \vv_{\varepsilon} \|_{H^1_0}^2$ (using $|
  \nse | = 1$ and the Sobolev embedding $H^1 \hookrightarrow L^4$).
  
  Combining \eqref{eq:Gexpansion} with the coercivity \eqref{eq:coerctransfer}
  yields
  \begin{equation}
    \label{eq:Glower} \mathcal{G}_{\varepsilon}  (\nse + \vv_{\varepsilon})
    -\mathcal{G}_{\varepsilon} (\nse) \geqslant (\alpha \tilde{\alpha} -
    C_{\xi} \varepsilon) \hspace{0.17em} \| \vv_{\varepsilon} \|_{H^1_0}^2
    \geqslant \tfrac{1}{2} \alpha \tilde{\alpha} \hspace{0.17em} \|
    \vv_{\varepsilon} \|_{H^1_0}^2,
  \end{equation}
  for $\varepsilon \in (0, \varepsilon_{\bullet})$ with
  $\varepsilon_{\bullet}$ further reduced if necessary.
  
  On the other hand, since $\mathcal{F}_{\varepsilon}
  =\mathcal{G}_{\varepsilon} +\mathcal{A}_{\varepsilon}$ and the dual
  functional is strictly non-negative ($\min \mathcal{A}_{\varepsilon}
  \geqslant 0$), we have the exact lower bound $\mathcal{G}_{\varepsilon}
  (\bm{m}^{\varepsilon}) \leqslant \min_{\gv \in \Vspace}
  \mathcal{F}_{\varepsilon} (\bm{m}^{\varepsilon}, \gv)$.
  Furthermore, because $(\bm{m}^{\varepsilon},
  \gv_{\ast}^{\varepsilon})$ globally minimizes $\mathcal{F}_{\varepsilon}$,
  and by applying Lemma~\ref{lem:energygap} to the competitor $\nse$, we
  obtain the tight upper bound:
  \begin{align}
    \label{eq:upperGseps} \mathcal{G}_{\varepsilon}
    (\bm{m}^{\varepsilon}) -\mathcal{G}_{\varepsilon}
    (\nse) & \hspace{0.27em} \leqslant \hspace{0.27em} \min_{\gv \in \Vspace}
    \mathcal{F}_{\varepsilon} (\bm{m}^{\varepsilon}, \gv)
    -\mathcal{G}_{\varepsilon} (\nse) \nonumber \\
    & \hspace{0.27em} \leqslant
    \hspace{0.27em} \min_{\gv \in \Vspace} \mathcal{F}_{\varepsilon} (\nse,
    \gv) -\mathcal{G}_{\varepsilon} (\nse) \hspace{0.27em} \leqslant
    \hspace{0.27em} C_{\xi} \varepsilon^2 .
  \end{align}
  Combining the lower bound \eqref{eq:Glower} with $\vv_{\varepsilon}
  =\bm{m}^{\varepsilon} - \nse$ and the upper bound
  \eqref{eq:upperGseps} yields
  \begin{equation}
    \tfrac{1}{2} \alpha \tilde{\alpha}  \hspace{0.17em}
    \|\bm{m}^{\varepsilon} - \nse \|_{H^1_0}^2
    \hspace{0.27em} \leqslant \hspace{0.27em} C_{\xi} \varepsilon^2,
  \end{equation}
  which immediately gives the desired convergence rate \eqref{eq:rate} with
  the optimized constant $K_{\xi} := \sqrt{2 C_{\xi} / (\alpha
  \tilde{\alpha})}$.
\end{proof}

\appendix

\section{ Non-dimensionalization}\label{sec:nondim}
We record the connection between the dimensional model of
Section~\ref{ssec:varprob} and the dimensionless functional $\mathcal{E}$ of
\eqref{eq:mainenergy}. The dimensional free energy is
\begin{equation}
  \label{eq:Fdim} \tilde{\mathcal{E}}(\widetilde{\nn},\widetilde{E}) = \frac{1}{2} 
  \int_{\tilde{\Omega}} \left[ K| \tilde{\nabla} \widetilde{\nn} |^2 -
  \varepsilon_0  \left( \varepsilon_{\perp} I + \varepsilon_a  \hspace{0.17em}
  \widetilde{\nn} \otimes \widetilde{\nn} \right)  \widetilde{\EE}
  \cdot \widetilde{\EE}\right] \hspace{0.17em} \mathrm{d} \tilde{x},
\end{equation}
with the  director $\widetilde{\nn} : \tilde{\Omega} \to \jM$ and the electric field $\widetilde{\EE}=-\nabla \tilde{U}$ with the corresponding electrostatic potential $\tilde{U} : \tilde{\Omega} \to \RR$. The potential is
coupled to the director as the unique solution to the dimensional Poisson
equation:
\begin{equation}
  \label{eq:Poisson_dim} \widetilde{\mathrm{div}}
  (\bm{\varepsilon}(\widetilde{\nn}) \tilde{\nabla}
  \tilde{U}) = 0 \quad \text{in } \tilde{\Omega}, \quad \tilde{U} =
  \tilde{U}_b  \quad \text{on } \partial \tilde{\Omega},
\end{equation}
where $\bm{\varepsilon} (\widetilde{\nn}) :=
\varepsilon_0  (\varepsilon_{\perp} I + \varepsilon_a  \widetilde{\nn} \otimes
\widetilde{\nn})$ is the dielectric tensor, and $\tilde{\nabla}$,
$\widetilde{\mathrm{div}}$ denote differential operators with respect to the
dimensional spatial variable $\tilde{x}$.

The physical units are $[\tilde{\mathcal{E}}_{\dim}] = \mathrm{J}$, $[K] =
\mathrm{J} \hspace{0.17em} \mathrm{m}^{- 1}$, $[\varepsilon_0] = \mathrm{F}
\hspace{0.17em} \mathrm{m}^{- 1}$, with $\varepsilon_{\perp}, \varepsilon_a$
dimensionless, $[\tilde{U}] = \mathrm{V}$, and spatial coordinates $\tilde{x}$
in meters. Let $\tilde{\ell} := \mathrm{diam} (\tilde{\Omega})$ and
$\tilde{V}$ be a characteristic applied voltage. We introduce the
dimensionless spatial variables $x = \tilde{x} / \tilde{\ell}$ in the rescaled
domain $\Omega := \tilde{\Omega} / \tilde{\ell}$, and define the dimensionless
fields:
\[ \nn (x) := \widetilde{\nn} (\tilde{\ell} x), \qquad u (x) := \tilde{V}^{-
   1}  \tilde{U} (\tilde{\ell} x), \qquad \xi_b (x) := \tilde{V}^{- 1} 
   \tilde{U}_b (\tilde{\ell} x) . \]
The chain rule gives $\tilde{\nabla} \widetilde{\nn} (\tilde{x}) =
\tilde{\ell}^{- 1} \nabla \nn (x)$, $\tilde{\nabla} \tilde{U} (\tilde{x}) =
(\tilde{V} / \tilde{\ell}) \nabla u (x)$, and $\mathrm{d} \tilde{x} =
\tilde{\ell}^3 \mathrm{d} x$. Substituting these variables into
\eqref{eq:Fdim} and dividing by the characteristic energy scale $\tfrac{1}{2}
\varepsilon_0 \varepsilon_{\perp}  \tilde{V}^2  \tilde{\ell}$ produces the
dimensionless energy functional:
\[ \mathcal{E}_\varepsilon (\nn) := \frac{\tilde{\mathcal{E}}_{\dim}
   (\widetilde{\nn})}{\tfrac{1}{2} \varepsilon_0 \varepsilon_{\perp} 
   \tilde{V}^2  \tilde{\ell}} = \int_{\Omega} \left[ \alpha | \nabla \nn |^2 -
   \Aeps (\nn) \nabla u \cdot \nabla u \right]  \dx, \]
with $\Aeps (\nn) = I + \eps  \hspace{0.17em} \nn \otimes \nn$.

Applying the exact same change of variables to \eqref{eq:Poisson_dim} and
factoring out the constants $\frac{\varepsilon_0 \varepsilon_{\perp} 
\tilde{V}}{\tilde{\ell}^2}$ yields the corresponding dimensionless Poisson
equation that the potential $u$ must satisfy:
\begin{equation}
  \label{eq:Poisson_nondim} \mathrm{div} (\Aeps (\nn) \nabla u) = 0 \quad
  \text{in } \Omega, \quad u = \xi_b  \quad \text{on } \partial \Omega .
\end{equation}
This scaling exposes two fundamental dimensionless parameters:
\begin{equation}
  \label{eq:rescparams} \alpha := \frac{K}{\varepsilon_0 \varepsilon_{\perp} 
  \tilde{V}^2}, \qquad \eps := \frac{\varepsilon_a}{\varepsilon_{\perp}} .
\end{equation}
The parameter $\eps$ records the dielectric anisotropy and is the small
parameter of Theorem~\ref{thm:thm2}; $\alpha$ is fixed and bounded away from
zero throughout the analysis. In some  typical liquid-crystal materials,
$\eps$ is of order $10^{- 1}$ or smaller as in MBBA (near $25^{\circ} C$) or PAA (near $122^{\circ}C$) but not for instance in 5CB (near $26^{\circ}C$), see Appendix $D$ in \cite{stewart2004}.

\section{ Uniform regularity and sufficient conditions for
coercivity}\label{sec:regularity}

This appendix establishes the uniform regularity estimate
\eqref{eq:uniformreg} required by Lemma~\ref{lem:coerctransfer}, and discusses
concrete geometric configurations in which Assumption~\ref{ass:coercivity} is
satisfied.

\subsection{Uniform \texorpdfstring{$W^{2, 2}$}{W{2, 2}} regularity}\label{ssec:W22}

\begin{theorem}[Uniform regularity]
  \label{thm:uniformW22est}Let $\Om \subset \RR^3$ be a bounded, simply
  connected domain of class $C^3$. Let $\nn_b \in C^{\infty}(\Om; \jM)$. 
  Then there exists $\varepsilon_{\star} > 0$ such that every global minimizer 
  $\nse$ of $\mathcal{G}_{\varepsilon}$ over $\mathcal{A}_{\nn_b}(\Om)$, 
  $\varepsilon \in (0, \varepsilon_{\star})$, satisfies the uniform bound
  \begin{equation}
    \label{eq:uniformW22} 
    \sup_{\varepsilon \in (0, \varepsilon_{\star})} \|\nse\|_{W^{2, 2}(\Om)} \leqslant C.
  \end{equation}
  Moreover, for any family $\{\nse\}_{\varepsilon \in (0, \varepsilon_{\star})}$ 
  of such global minimizers, there exists a subsequence (not relabeled) and a 
  global minimizer $\ns \in W^{2, 2}(\Om) \cap W^{1, \infty}(\Om)$ of 
  $\mathcal{G}_0$ such that
  \begin{equation}
    \label{eq:strongH1} 
    \|\nse - \ns\|_{H^1(\Om)} \to 0 \qquad \text{as } \varepsilon \to 0.
  \end{equation}
  Consequently, by interpolation, $\nse \to \ns$ strongly in $W^{1, p}(\Om)$
  for every $p \in [2, 6)$, which guarantees the strong $W^{1, 4}(\Om)$
  convergence required to satisfy hypothesis \eqref{eq:uniformreg} of
  Lemma~\ref{lem:coerctransfer}.
\end{theorem}

\begin{proof}[Sketch of proof]
By the $\Gamma$-convergence result established in Section~\ref{ssec:gamma}, there exists a global minimizer $\ns \in \mathcal{A}_{\nn_b}(\Om)$ of $\mathcal{G}_0$ such that $\|\nse - \ns\|_{H^1(\Om)} \to 0$ as $\varepsilon \to 0$. From the classical regularity theory for harmonic maps into $\jM$ with smooth boundary data $\nn_b \in C^{\infty}(\Om; \jM)$ on a $C^3$ domain, the limit minimizer enjoys the regularity $\ns \in W^{2,2}(\Om) \cap W^{1,\infty}(\Om)$ stated in the theorem.

To obtain uniform higher-order estimates for $\nse$, we proceed in two steps: we first establish uniform H\"older continuity up to the boundary by adapting the partial regularity theory of Hardt--Kinderlehrer--Lin~\cite{Hardt1986}, and then we derive uniform $W^{2,2}(\Om)$ bounds via linear elliptic difference-quotient techniques~\cite{Borchers}.

The partial regularity results in~\cite{Hardt1986} apply to general functionals of the form (see \cite[pp.~547 and 551]{Hardt1986}):
\begin{equation}
  \label{eq:HKL_functional}
  \underbrace{\int_\Om |\nabla \nn|^2 \dx}_{:=W(\nabla \nn)} + \underbrace{\int_\Om \left( \bm{A}[\nn] : \nabla \nn + \bm{B}\nn \cdot \nn + \bm{c} \cdot \nn \right) \dx}_{:=\mathcal{F}(\nabla \nn, \nn)},
\end{equation}
where $\bm{A}$, $\bm{B}$, and $\bm{c}$ are bounded coefficient fields on $\Om$. Our functional $\mathcal{G}_\varepsilon(\nn)$ can be cast exactly into this form with $\bm{A} \equiv \bm{0}$, where the lower-order coefficient tensors $\bm{B}$ and $\bm{c}$ depend on the background field $\xis$ and its spatial gradient $\nabla \xis$. Since we assume $\xis \in W^{1,\infty}(\Om)$, both $\xis$ and $\nabla \xis$ are essentially bounded on $\Om$. This Lipschitz regularity guarantees that $\bm{B}$ and $\bm{c}$ belong to $L^\infty(\Om)$, allowing us to bound the supremum
\begin{equation}
  \label{eq:F_supremum}
  \kappa := \sup_{\Om} \left( |\bm{A}| + |\bm{B}| + |\bm{c}| \right)
\end{equation}
by a finite constant independent of $\varepsilon$ for all $\varepsilon$ within a bounded interval $(0, \varepsilon_0)$.

To invoke Corollary~3.5 in \cite{Hardt1986}, which yields quadratic energy decay and local H\"older continuity via Morrey's lemma, we must first verify that the small-energy hypothesis holds uniformly on any ball $B_R(a) \subset \Om$:
\begin{equation}
  \label{eq:uniformDeltaR}
  \int_{B_{R}(a)} |\nabla \nse|^2 \dx \leqslant \delta R,
\end{equation}
where $\delta > 0$ is a small structural threshold. 

Because $\ns \in W^{2,2}(\Om) \hookrightarrow W^{1,6}(\Om)$ by Sobolev embedding in dimension $n=3$, applying H\"older's inequality with exponents $p=3$ and $q=3/2$ yields
\begin{equation}
  \label{eq:small_energy_holder}
  \int_{B_R(a)} |\nabla \ns|^2 \dx \leqslant \|\nabla \ns\|_{L^6(\Om)}^2 |B_R(a)|^{\frac{2}{3}} \leqslant C \|\nabla \ns\|_{L^6(\Om)}^2 R^2.
\end{equation}
Since the right-hand side scales quadratically as $O(R^2)$, for any given $\delta > 0$ we can choose a radius $R > 0$ small enough so that $C \|\nabla \ns\|_{L^6(\Om)}^2 R \leqslant \frac{\delta}{2}$, which guarantees
\begin{equation}
  \label{eq:small_energy_ns}
  \int_{B_R(a)} |\nabla \ns|^2 \dx \leqslant \frac{\delta}{2}R,
\end{equation}
uniformly with respect to the center $a \in \overline{\Om}$.

By the strong convergence $\|\nse - \ns\|_{H^1(\Om)} \to 0$, there exists $\varepsilon_{\star} \in (0, \varepsilon_0)$ such that for all $\varepsilon \in (0, \varepsilon_{\star})$, we have
\begin{equation}
  \label{eq:small_energy_diff}
  \int_{B_R(a)} |\nabla(\nse - \ns)|^2 \dx \leqslant \frac{\delta}{2}R,
\end{equation}
again uniformly in $a$. Combining \eqref{eq:small_energy_ns} and \eqref{eq:small_energy_diff} via the triangle inequality confirms \eqref{eq:uniformDeltaR}. 

With the small-energy hypothesis verified uniformly in $\varepsilon$, Corollary~3.5 in \cite{Hardt1986} yields the decay estimate for all $0 < r < R$:
\begin{equation}
  \label{eq:HKL_decay}
  \int_{B_{r}(a)} |\nabla \nse|^2 \dx \leqslant  \frac{\max\{\delta, \eta \kappa^2 R^2\}}{\theta^{2}r^2 R},
\end{equation}
where $\theta, \eta > 0$ are structural constants. Since the right-hand side scales as $O(r^2)$ in dimension $n=3$, Morrey's lemma immediately implies the interior H\"older estimate $\|\nse\|_{C^{0,1/2}(K)} \leqslant C$ on any compact subset $K \Subset \Om$. Furthermore, because the boundary $\partial\Om$ is of class $C^3$ and the Dirichlet boundary data satisfies $\nn_b \in C^{\infty}(\Om; \jM)$, the reflection arguments in Section~5 of \cite{Hardt1986} apply up to the boundary, yielding the uniform global H\"older bound
\begin{equation}
  \label{est:nholder}
  \|\nse\|_{C^{0,1/2}(\overline{\Om})} \leqslant C.
\end{equation}

Having established uniform H\"older continuity on the closed domain, we derive uniform $W^{2,2}(\Om)$ estimates by adapting the interior regularity argument of Borchers~\cite{Borchers} to the boundary in two standard steps.

\smallskip

\noindent\textit{Step 1 (Flat boundary case).}
Consider a half-ball $K = B_\rho \cap \{x_3 \geqslant 0\}$ where $\nse$ is a critical point of $\mathcal{G}_\varepsilon$ satisfying \eqref{est:nholder}. By taking tangential difference quotients in the Euler--Lagrange system as in \cite{Borchers}, we obtain uniform $L^2$ bounds for all second-order derivatives except the normal second derivative $\partial_{33}\nse$. Since the Euler--Lagrange system is non-degenerate elliptic, we can isolate $\partial_{33}\nse$ and express it algebraically in terms of the tangential second derivatives and lower-order terms. This yields $\|\nse\|_{W^{2,2}(K)} \leqslant C$, with $C$ independent of $\varepsilon$.

\smallskip

\noindent\textit{Step 2 (General boundary case).}
Since $\Om$ is simply connected and of class $C^3$, we can locally flatten $\partial\Om$ via a $C^3$ diffeomorphism. Under this change of variables, the Euler--Lagrange system transforms into a uniformly elliptic system whose leading coefficients are of class $C^2$. The difference-quotient arguments from Step~1 extend to this variable-coefficient setting by classical linear elliptic boundary regularity theory. Covering the compact set $\overline{\Om}$ by finitely many interior balls and boundary charts yields the desired global uniform estimate
\begin{equation}
  \label{eq:final_W22_bound}
  \sup_{\varepsilon \in (0, \varepsilon_{\star})} \|\nse\|_{W^{2,2}(\Om)} \leqslant C.
\end{equation}
Finally, by the reflexivity of $W^{2,2}(\Om)$ and the uniform bound \eqref{eq:final_W22_bound}, any family $\{\nse\}_{\varepsilon \in (0, \varepsilon_{\star})}$ admits a weakly convergent subsequence $\nse \rightharpoonup \ns$ in $W^{2,2}(\Om)$. Since we already know that $\nse \to \ns$ strongly in $H^1(\Om)$, standard Gagliardo--Nirenberg interpolation inequalities imply that the convergence is strong in $W^{1,p}(\Om)$ for every $2 \leqslant p < 6$. In particular, taking $p = 4$ yields the strong $W^{1,4}(\Om)$ convergence required to satisfy hypothesis~\eqref{eq:uniformreg} of Lemma~\ref{lem:coerctransfer}.
\end{proof}

\subsection{Sufficient conditions for the coercivity
assumption}\label{ssec:coercivity-suff}We record two configurations in which
Assumption~\ref{ass:coercivity} is met.

\begin{proposition}[Open hemisphere condition]
  \label{lem:hemisphere}Let $\Om \subset \RR^N$ ($N \geqslant 2$) be a bounded
  smooth enough domain (e.g., of class $C^{1, \alpha}$). Let $\ns \in H^1 (\Om, \jM)$
  be a global minimizer of the Dirichlet energy with boundary data $\nn_b$
  taking values in an open hemisphere of $\jM$. Assume further that $\ns \in
  W^{1, N} (\Om)$ if $N \geqslant 3$, or $\ns \in W^{1, p_{\ast}} (\Om)$ for
  some $p_{\ast} > 2$ if $N = 2$. Then $\ns$ takes values in the same open
  hemisphere, and Assumption~\ref{ass:coercivity} holds.
\end{proposition}

\begin{remark}[On the $C^{1, \alpha}$ domain requirement]
  The $C^{1, \alpha}$-regularity assumption on the domain in
  Proposition~\ref{lem:hemisphere} is geometrically necessary in order to align with
  the boundary regularity result of Schoen--Uhlenbeck in
  \cite{schoen_uhlenbeck_1983_boundary}, whose proof relies on $C^{1,
  \alpha}$ coordinate charts to flatten the boundary. For merely Lipschitz
  domains, H\"older continuity up to the boundary can instead be guaranteed
  only under the stronger assumption $\ns \in W^{1, p} (\Om)$ with $p > N$, an
  assumption that in our framework is imposed only in the two-dimensional
  case.
\end{remark}

\begin{proof}
  \emph{Step 1. Hemisphere containment.} The closed hemisphere is
  geodesically convex. Hence, by the convex-hull property for minimizing
  harmonic maps (see, e.g., \cite[§5]{Hardt1986}), the minimizer $\ns$
  takes values in the closed hemisphere. Moreover, the strict interior
  condition is preserved throughout the domain by the strong maximum
  principle, and up to the boundary by the H\"older continuity of $\ns$
  provided by the Schoen--Uhlenbeck regularity theorem
  \cite{schoen_uhlenbeck_1983_boundary}. Consequently, the image of $\ns$
  remains uniformly separated from the equator. Geometrically, this means that
  there exist a fixed pole $\bm{e} \in \jM$ and a
  constant $c > 0$ such that $u := \ns \cdot \bm{e}
  \geqslant c$ on $\overline{\Om}$.
   \smallskip
  
\noindent
  \emph{Step 2. Strict positivity via a Hardy-type identity.} As
  a harmonic map, $\ns$ weakly solves $- \Delta \ns = | \nabla \ns |^2 \ns$.
  Taking the dot product with $\bm{e}$ yields the
  eigenvalue equation $- \Delta u = | \nabla \ns |^2 u$. Let $\vv \in H^1_0
  (\Om, \mathbb{R}^3) \setminus \{0\}$ be an admissible variation satisfying
  the pointwise constraint $\vv \cdot \ns = 0$. By the Schoen-Uhlenbeck
  continuity established in Step 1, $u \geqslant c > 0$ on $\bar{\Omega}$.
  Thus, $1 / u \in L^{\infty} (\Omega)$, and the quotient
  $\bm{w} = \vv / u$ is a well-defined element of $H^1_0
  (\Omega, \mathbb{R}^3)$. Expanding $\vv = u\bm{w}$ by
  the product rule gives $\nabla \vv = u \nabla
  \bm{w}+\bm{w} \otimes \nabla u$.
  Squaring this and expanding the cross-term yields:
  \[ | \nabla \vv |^2 \hspace{0.27em} = \hspace{0.27em} u^2  | \nabla
     \bm{w} |^2 + | \nabla u|^2
     |\bm{w}|^2 + \frac{1}{2} \nabla (u^2) \cdot \nabla
     (|\bm{w}|^2) . \]
  Integrating over $\Om$ and integrating the last term by parts, the gradient
  terms $| \nabla u|^2$ cancel exactly:
  \[ \int_{\Om} | \nabla \vv |^2 \dx \hspace{0.27em} = \hspace{0.27em}
     \int_{\Om} u^2  | \nabla \bm{w}|^2 \dx - \int_{\Om}
     u \Delta u |\bm{w}|^2 \dx . \]
  Substituting $- u \Delta u = | \nabla \ns |^2 u^2$ and $u^2
  |\bm{w}|^2 = | \vv |^2$, we find that the second
  variation is explicitly positive:
  \[ \mathcal{L}_0 [\vv] \hspace{0.27em} := \hspace{0.27em} \int_{\Om} |
     \nabla \vv |^2 \dx - \int_{\Om} | \nabla \ns |^2 | \vv |^2 \dx
     \hspace{0.27em} = \hspace{0.27em} \int_{\Om} u^2  | \nabla (\vv / u) |^2
     \dx \hspace{0.27em} \geqslant \hspace{0.27em} 0. \]
  Since $u \geqslant c > 0$, if $\mathcal{L}_0 [\vv] = 0$ then necessarily
  $\nabla (\vv / u) = 0$, meaning $\vv = Cu$. Since $\vv \in H^1_0 (\Om)$
  vanishes on the boundary while $u$ is strictly positive there, $C$ must be
  $0$. Thus, $\mathcal{L}_0 [\vv] > 0$ strictly for all non-zero admissible
  variations.
   \smallskip
  
\noindent
  \emph{Step 3. Lower bound.} We want to prove the existence of a
  constant $M > 0$ such that the following algebraic lower bound holds:
  \begin{equation}
    \label{eq:trivialcoerc} \mathcal{L}_0 [\vv] \hspace{0.27em} \geqslant
    \hspace{0.27em} \frac{1}{2} \left| \nabla \vv \right|^2_{L^2} - M \left\|
    \vv \right\|^2_{L^2} .
  \end{equation}
  For that, we introduce the cutoff decomposition $| \nabla \ns |^2 = f_M +
  g_M$, with 
  \begin{equation}
  f_M := | \nabla \ns |^2 \chi_{\{| \nabla \ns |^2 \leqslant M\}}
  \quad \text{and}\quad g_M := | \nabla \ns |^2 \chi_{\{| \nabla \ns |^2 > M\}}. 
 \end{equation}  
  Also, we
  define an integrability exponent $q$ depending on the dimension $N$. If $N
  \geqslant 3$, let $q = N / 2$. If $N = 2$, let $q = p_{\ast} / 2 > 1$. In
  both cases, our regularity hypotheses guarantee that $g_M$ belongs to $L^q
  (\Om)$. Let $q' = q / (q - 1) < \infty$ be the H\"older conjugate of $q$. By
  the Sobolev embedding theorem, $H^1_0 (\Om)$ embeds continuously into $L^{2
  q'} (\Om)$ for any $N \geqslant 2$. Applying H\"older's inequality yields:
  \begin{equation}
    \label{eq:cutoff_bound_unified} \int_{\Om} | \nabla \ns |^2 | \vv |^2 \dx
    \hspace{0.27em} \leqslant \hspace{0.27em} M \left| \vv \right|_{L^2}^2 +
    \|g_M \|_{L^q} \| \vv \|_{L^{2 q'}}^2 \leqslant \hspace{0.27em} M \left\|
    \vv \right\|^2_{L^2} + C \|g_M \|_{L^q} \| \nabla \vv \|_{L^2}^2 .
  \end{equation}
  By the dominated convergence theorem, $\|g_M \|_{L^q} \to 0$ as $M \to
  \infty$. Fixing $M$ large enough so that $C \|g_M \|_{L^q} \leqslant 1 / 2$,
  we subtract \eqref{eq:cutoff_bound_unified} from $\| \nabla \vv \|_{L^2}^2$
  to get \eqref{eq:trivialcoerc}.
   \smallskip
  
\noindent
  \emph{Step 4. Upgrading to quadratic coercivity.} Consider the
  constrained $L^2$-minimization problem:
  \[ \alpha_{\ast} \hspace{0.27em} := \hspace{0.27em} \inf \left\{
     \mathcal{L}_0 [\vv] : \vv \in H^1_0 (\Om, \mathbb{R}^3), \hspace{0.17em}
     \vv \cdot \ns = 0 \text{ a.e.}, \hspace{0.17em} \| \vv \|_{L^2 (\Om)} = 1
     \right\} . \]
  Let $\{ \vv_k \}_{k \in \NN}$ be a minimizing sequence. By
  \eqref{eq:trivialcoerc}, the sequence $\{ \vv_k \}_{k \in \NN}$ is bounded
  in $H^1_0 (\Om)$; hence, up to the extraction of a subsequence, there exists
  $\vv_{\ast} \in H^1_0$ such that $\vv_k \rightharpoonup \vv_{\ast}$ weakly
  in $H^1_0$ and strongly in $L^2$. Moreover, the same cutoff decomposition
  used in the previous step shows that the quadratic form $\mathcal{L}_0$ is
  weakly lower semicontinuous (indeed, one only needs to observe that $\int
  f_M | \vv_k |^2 \to \int f_M | \vv_{\ast} |^2$ and $\|g_M \|_{L^q} \to 0$ as
  $M \to \infty$). It follows that the infimum $\alpha_{\ast}$ is attained at
  $\vv_{\ast}$. Since $\| \vv_{\ast} \|_{L^2} = 1$, Step~2 guarantees
  $\alpha_{\ast} =\mathcal{L}_0 [\vv_{\ast}] > 0$. By homogeneity, we
  therefore obtain
  \begin{equation}
    \label{eq:L2coerc} \mathcal{L}_0 [\vv] \hspace{0.27em} \geqslant
    \hspace{0.27em} \alpha_{\ast} \| \vv \|_{L^2 (\Om)}^2 \qquad \text{for all
    admissible } \vv .
  \end{equation}
  To upgrade this to strict $H^1_0$-coercivity, it is sufficient to take a
  convex combination of \eqref{eq:L2coerc} and \eqref{eq:trivialcoerc} with
  parameter $\theta \in (0, 1)$, which gives
  \[ \mathcal{L}_0 [\vv] \hspace{0.27em} \geqslant \hspace{0.27em} \frac{1 -
     \theta}{2} \| \nabla \vv \|_{L^2 (\Om)}^2 + (\theta \alpha_{\ast} - (1 -
     \theta) M) \| \vv \|_{L^2 (\Om)}^2 . \]
  Choosing $\theta$ sufficiently close to $1$ ensures the coefficient of $\|
  \vv \|_{L^2 (\Om)}^2$ is strictly positive. Setting $\alpha_0 := (1 -
  \theta) / 2 > 0$ we obtain the desired coercivity estimate $\mathcal{L}_0
  [\vv] \geqslant \alpha_0 \| \nabla \vv \|_{L^2 (\Om)}^2$.
\end{proof}

While the general abstract theory developed in this paper requires smooth
domains to invoke classical boundary regularity, our foundational coercivity
assumption (Assumption~\ref{ass:coercivity}) is explicitly verifiable in the
canonical physical geometries. In liquid crystal applications, the most
prevalent domain is the parallel-plate cell (e.g., a twisted nematic device).
To rigorously study field-induced instabilities such as the Fréedericksz
transition in such devices, one must first prove that the unperturbed planar
configuration is strictly isolated and coercive. Because a rectangular box is
merely a Lipschitz domain, the general boundary regularity from
Lemma~\ref{lem:hemisphere} cannot be directly applied. However, the following
lemma demonstrates that by restricting the boundary conditions to an open
angular sector $(0, \pi)$, we can use an exact geometric projection to bypass
the need for boundary regularity. The director is geometrically trapped in the
plane, and the linear twist interpolant is the unique, strictly coercive
global minimizer.

\begin{lemma}[Parallel-plate geometry]
  \label{lem:parallelplates}Let $\Om = (- L, L)^2 \times (0, h)$ with boundary
  data $\nn_b (x, y, z)$ defined such that on the top ($z = h$) and bottom ($z
  = 0$) plates it takes the uniform planar values $\nn_b^{\pm} = (\cos
  \theta^{\pm}, \sin \theta^{\pm}, 0)$, with $\theta^{\pm} \in (0, \pi)$
  constants, and matching linearly on the lateral boundaries. Then the planar
  director
  \begin{equation}
    \overline{\nn} (x, y, z) := (\cos \theta (z), \sin \theta (z), 0), \qquad
    \theta (z) := \theta^- + \frac{\theta^+ - \theta^-}{h}  \hspace{0.17em} z,
  \end{equation}
  is the unique global minimizer of $\mathcal{G}_0$ in $\mathcal{A}_{\nn_b}$,
  and Assumption~\ref{ass:coercivity} holds at $\overline{\nn}$.
\end{lemma}

\begin{proof} We subdivide the proof into three steps.

 \noindent \emph{Step 1. Confinement to the plane.} Let $\ns \in H^1 (\Om, \jM)$ be a
global energy minimizer. Since the boundary data $\nn_b$ takes values in the
open hemisphere $y > 0$ (as $\theta^{\pm} \in (0, \pi)$), Proposition~\ref{lem:hemisphere}
guarantees that $\ns \cdot \bm{e}_y \geqslant c > 0$ in $\Om$. In particular, the
equatorial projection $P (\nn) := \nn - (\nn \cdot \bm{e}_z) \bm{e}_z$ satisfies
$|P (\ns)| \geqslant c > 0$ everywhere in $\Om$.
Consequently, the normalized projection $\widehat{\nn} := {P (\ns)}/{|P (\ns)|}$
is a well-defined element of $\mathcal{A}_{\nn_b} (\Om)$. Applying the standard
spherical gradient inequality, $|\nabla \widehat{\nn}|^2 \leqslant |\nabla \ns|^2$
holds pointwise with equality if and only if $n_z \equiv 0$. The global
minimality of $\ns$ therefore forces $n_z \equiv 0$ almost everywhere,
rendering the minimizer strictly planar.

   \smallskip
  
\noindent
  \emph{Step 2. Uniqueness of the linear interpolant.} Since $\ns$ is planar and $\Om$ is simply connected, it
  admits a global lifting $\ns = (\cos \phi, \sin \phi, 0)$ for a scalar phase
  $\phi \in H^1 (\Om, \RR)$. The Dirichlet energy reduces to the standard
  scalar functional
  \begin{equation}
    \int_{\Om} | \nabla \ns |^2 \dx \hspace{0.27em} = \hspace{0.27em}
    \int_{\Om} | \nabla \phi |^2 \dx .
  \end{equation}
  By strict convexity, its unique minimizer subject to the boundary conditions
  $\phi (x, y, h) = \theta^+$ and $\phi (x, y, 0) = \theta^-$ is the
  one-dimensional harmonic extension $\phi (x, y, z) = \theta (z)$. Hence,
  $\ns \equiv \overline{\nn}$.
   \smallskip
  
\noindent
  \emph{Step 3. Strict quadratic coercivity.} We evaluate the second variation of the energy at $\overline{\nn}$. For the
  planar state $\overline{\nn}$, the gradient squared is uniform: $| \nabla
  \overline{\nn} |^2 = | \theta' (z) |^2 = q^2$, where we define the constant
  gradient
  \begin{equation}
    q := (\theta^+ - \theta^-) / h.
  \end{equation}
  The second variation $\mathcal{L}_0 [\vv]$ for any admissible tangent
  variation $\vv \in H^1_0 (\Om, \RR^3)$ satisfying $\vv \cdot \overline{\nn}
  = 0$ is
  \begin{equation}
    \mathcal{L}_0 [\vv] \hspace{0.27em} = \hspace{0.27em} \int_{\Om} | \nabla
    \vv |^2 \dx - q^2  \int_{\Om} | \vv |^2 \dx .
  \end{equation}
  Since $\vv (x, y, \cdot) \in H^1_0 (0, h ; \RR^3)$ for almost every $(x,
  y)$, the one-dimensional Poincaré inequality yields
  \begin{equation}
    \int_0^h \left| \vv \right|^2 \mathrm{d} z \leqslant \left( \frac{h}{\pi}
    \right)^2 \int_0^h \left| \partial_z \vv \right|^2 \mathrm{d} z \leqslant
    \left( \frac{h}{\pi} \right)^2 \int_0^h \left| \nabla \vv \right|^2
    \mathrm{d} z
  \end{equation}
  Integrating this bound over the horizontal cross-section $(- L, L)^2$ yields
  the global estimate $\| \vv \|_{L^2 (\Om)}^2 \leqslant (h / \pi)^2 \| \nabla
  \vv \|_{L^2 (\Om)}^2$. Substituting this into the second variation gives:
  \begin{equation}
    \mathcal{L}_0 [\vv] \hspace{0.27em} \geqslant \hspace{0.27em} \left[ 1 -
    \left( \frac{qh}{\pi} \right)^2 \right]  \int_{\Om} | \nabla \vv |^2 \dx .
  \end{equation}
  By hypothesis, $\theta^+, \theta^- \in (0, \pi)$, meaning the total twist
  satisfies $| \theta^+ - \theta^- | < \pi$. This strictly ensures $qh / \pi <
  1$. Setting $\alpha_0 := 1 - (qh / \pi)^2 > 0$ confirms the coercivity bound
  $\mathcal{L}_0 [\vv] \geqslant \alpha_0 \| \nabla \vv \|_{L^2 (\Om)}^2$,
  concluding the proof.
\end{proof}

Finally, we demonstrate that strict quadratic coercivity
(Assumption~\ref{ass:coercivity}) is not automatically guaranteed for global
minimizers. If the boundary data is not confined to an open hemisphere, the
energy landscape of a global minimizer can possess zero-modes. This confirms
the geometric necessity of Lemma~\ref{lem:hemisphere}.

\begin{proposition}[Failure of coercivity for boundary twists $\Theta > \pi$]
  \label{prop:instability}Let $\Om = B_R \times (0, L) \subset \RR^3$ be a
  finite cylinder. For any prescribed twist rate $q > 0$, let $\ns (z) = (\sin
  (qz), 0, \cos (qz))$ be the planar director field satisfying the
  corresponding Dirichlet boundary condition $\nn_b = \ns |_{\partial \Om}$
  with total twist $\Theta = qL$. There exists a critical threshold
  \begin{equation}
    \Theta_c := \sqrt{z_0^2  (L / R)^2 + \pi^2} > \pi,
  \end{equation}
  where $z_0 \approx 2.4048$ is the first positive root of the Bessel function
  $J_0$, such that:
  \begin{enumerate}
    \item[(i)] If $\Theta < \Theta_c$, the planar state $\ns$ is the unique
    global minimizer of the Dirichlet energy, and its second variation
    $\mathcal{L}_0$ is strictly positive definite.
    
    \item[(ii)] If $\Theta = \Theta_c$, $\ns$ remains a global minimizer, but
    there exists a non-trivial tangent variation
    $\bm{u}_{\perp} \in H^1_0 (\Om, \RR^3)$ with
    $\bm{u}_{\perp} \cdot \ns = 0$ such that
    $\mathcal{L}_0 [\bm{u}_{\perp}] = 0$. Consequently,
    strict coercivity (Assumption~\ref{ass:coercivity}) fails.
  \end{enumerate}
  In particular, since $\Theta_c > \pi$, the existence of such zero-modes is
  unconditionally ruled out whenever the boundary data is confined to an open
  hemisphere ($\Theta < \pi$).
\end{proposition}

\begin{proof}
  Let $\Om = B_R \times (0, L)$ and fix $q > 0$. The purely $z$-dependent map
  $\ns (x, y, z) := (\sin (qz), 0, \cos (qz))$ satisfies $| \nabla \ns |^2 =
  q^2$ and $- \Delta \ns = q^2 \ns = | \nabla \ns |^2 \ns$, confirming it is a
  weakly harmonic map.
   \smallskip
  
\noindent
  \emph{Step 1. Global minimality via exact energy decomposition.} Let $\nn
  \in H^1 (\Om, \jM)$ be an arbitrary competitor satisfying $\nn = \ns$ on
  $\partial \Om$, and define the difference $\vv := \nn - \ns \in H^1_0 (\Om,
  \RR^3)$. The pointwise spherical constraint $| \nn |^2 = | \ns + \vv |^2 =
  1$ implies $2 \ns \cdot \vv = - | \vv |^2$ almost everywhere.
  
  Expanding the Dirichlet energy of $\nn$ and integrating by parts yields
  \begin{align*}
    \frac{1}{2}  \int_{\Om} | \nabla \nn |^2 \dx & = \frac{1}{2}  \int_{\Om} |
    \nabla \ns |^2 \dx + \int_{\Om} \nabla \ns : \nabla \vv \dx + \frac{1}{2} 
    \int_{\Om} | \nabla \vv |^2 \dx\\
    & = \frac{1}{2}  \int_{\Om} | \nabla \ns |^2 \dx - \int_{\Om} \Delta \ns
    \cdot \vv \dx + \frac{1}{2}  \int_{\Om} | \nabla \vv |^2 \dx .
  \end{align*}
  Using $- \Delta \ns = q^2 \ns$ and the spherical constraint, the cross-term
  simplifies exactly to
  \begin{equation}
    - \int_{\Om} \Delta \ns \cdot \vv \dx = q^2  \int_{\Om} \ns \cdot \vv \dx
    = - \frac{q^2}{2}  \int_{\Om} | \vv |^2 \dx .
  \end{equation}
  Thus, the excess energy of any competitor is exactly represented by the
  quadratic form:
  \begin{equation}
    \label{eq:exact_energy_identity} \int_{\Om} | \nabla \nn |^2 \dx -
    \int_{\Om} | \nabla \ns |^2 \dx = \int_{\Om} \left( | \nabla \vv |^2 - q^2
    | \vv |^2 \right) \dx .
  \end{equation}
  Let $\lambda_1$ be the principal Dirichlet eigenvalue of the negative
  Laplacian on $\Om$. By the Poincaré inequality, $\int | \nabla \vv |^2 \dx
  \geqslant \lambda_1  \int | \vv |^2 \dx$ for all $\vv \in H^1_0 (\Om,
  \RR^3)$. For $q^2 \leqslant \lambda_1$, the excess energy is non-negative,
  proving $\ns$ is a global minimizer. For $q^2 < \lambda_1$, the excess
  energy is strictly positive for $\vv \neq 0$, rendering $\ns$
  unique.
   \smallskip
  
\noindent
  \emph{Step 2. Failure of strict coercivity.} The strict
  coercivity condition (Assumption~\ref{ass:coercivity}) requires the second
  variation $\mathcal{L}_0 [\bm{u}]$ to be strictly
  positive for all non-trivial tangent variations $\bm{u}
  \in H^1_0 (\Om, \RR^3)$ satisfying $\bm{u} \cdot \ns =
  0$. The second variation at $\ns$ evaluates to
  \begin{equation}
    \mathcal{L}_0 [\bm{u}] = \int_{\Om} (| \nabla
    \bm{u}|^2 - q^2 |\bm{u}|^2) \dx
    .
  \end{equation}
  Define the orthogonal perturbation
  $\bm{u}_{\perp} (x, y, z) = (0, \psi (r, z),
  0)$, where $\psi \in H^1_0 (\Om)$ is the principal eigenfunction of the
  cylinder:
  \begin{equation}
    \psi (r, z) = J_0 \left( \frac{z_0 r}{R} \right) \sin \left( \frac{\pi
    z}{L} \right), \qquad \lambda_1 = \frac{z_0^2}{R^2} + \frac{\pi^2}{L^2} .
  \end{equation}
  Because $\ns$ lies entirely in the $xz$-plane,
  $\bm{u}_{\perp} \cdot \ns = 0$ everywhere, making
  $\bm{u}_{\perp}$ an admissible tangent vector. At the
  threshold $q^2 = \lambda_1$, we obtain $\mathcal{L}_0
  [\bm{u}_{\perp}] = 0$. Thus, $\ns$ remains a
  global minimizer, but its energy landscape possesses a zero-mode, violating
  strict coercivity.
   \smallskip
  
\noindent
  \emph{Step 3. Geometric twist threshold.} The boundary data
  traces an equatorial arc of length $\Theta = qL$. At the critical threshold
  $q = \sqrt{\lambda_1}$ where strict coercivity fails, this total twist is
  \begin{equation}
    \Theta_c = L \sqrt{\frac{z_0^2}{R^2} + \frac{\pi^2}{L^2}} = \sqrt{z_0^2
    \left( \frac{L}{R} \right)^2 + \pi^2} .
  \end{equation}
  Since $z_0 > 0$ and $L / R > 0$, we strictly have $\Theta_c > \pi$.
  Consequently, strict coercivity can only degenerate if the image of the
  boundary data strictly exceeds a great semicircle, demonstrating the
  geometric sharpness of the hemisphere condition ($\Theta < \pi$).
\end{proof}

\section{Bifurcation of accumulation points} \label{app:bifurcation}
This section aims to rigorously establish the claim made in Remark~\ref{rmk:bifurcation} regarding the possible bifurcation of the asymptotic limits for minimizers of the exact physical energy $\mathcal{E}_{\varepsilon}$ and its local approximation $\mathcal{G}_{\varepsilon}$. 

Let $\Omega = \omega \times (0, \delta) \subset \mathbb{R}^3$, where $\omega := (0, L)^2$ and $L, \delta > 0$, be a square-based cylindrical domain. We consider the fundamental planar state on $\Omega$,
\begin{equation}
  \boldsymbol{n}_{\mathrm{pl}}(x, y, z) = (\cos \phi(x), \sin \phi(x), 0),
\end{equation}
where $\phi(x) = \Theta(x / L - 1 / 2)$, with $\Theta > 0$ denoting the total twist angle across the longitudinal axis. We prescribe strictly planar Dirichlet boundary conditions to anchor the director $\boldsymbol{n}$ to the equator of the target manifold $\mathbb{S}^2$ on the entire boundary $\partial \Omega$:
\begin{equation}
  \boldsymbol{n}_b = \boldsymbol{n}_{\mathrm{pl}} \quad \text{on } \partial \Omega.
\end{equation}
Electrostatic coupling is incorporated by prescribing the boundary potential $\xi_b(x, y, z) = z$ on $\partial \Omega$. The associated harmonic extension generates a uniform background electric field
\begin{equation}
  \boldsymbol{E}_0 = - \nabla z = -\boldsymbol{e}_z.
\end{equation}

In this setting, the local approximation $\mathcal{G}_{\varepsilon}$ takes the form
\begin{equation}
  \mathcal{G}_{\varepsilon}(\boldsymbol{n}) = \int_{\Omega} |\nabla \boldsymbol{n}|^2 \, dx - \varepsilon \int_{\Omega} (\boldsymbol{n} \cdot \boldsymbol{e}_z)^2 \, dx.
\end{equation}
It is a straightforward algebraic verification that, for every $\varepsilon > 0$, the planar state $\boldsymbol{n}_{\mathrm{pl}}$ satisfies the Euler--Lagrange equation
\begin{equation}
  - \Delta \boldsymbol{n} - \varepsilon (\boldsymbol{e}_z \otimes \boldsymbol{e}_z) \boldsymbol{n} = \left( |\nabla \boldsymbol{n}|^2 - \varepsilon (\boldsymbol{n} \cdot \boldsymbol{e}_z)^2 \right) \boldsymbol{n},
\end{equation}
and is therefore \textit{always} an equilibrium configuration (a critical point) for the local approximation $\mathcal{G}_{\varepsilon}$, regardless of the applied voltage. However, as we will demonstrate in the proof of the main result below, if the imposed boundary twist exceeds the critical threshold
\begin{equation}
  \Theta > \Theta_c, \quad \text{where} \quad \Theta_c := \pi \sqrt{2 + \frac{L^2}{\delta^2}},
\end{equation}
the planar state becomes energetically unstable. Moreover, it is the unique, absolute global minimizer of the unperturbed energy $\mathcal{G}_0$ restricted to the class $\mathcal{A}_{\mathrm{pl}} = \{\boldsymbol{n} \in \mathcal{A}_{\boldsymbol{n}_b}(\Omega) : (\boldsymbol{n} \cdot \boldsymbol{e}_z) \equiv 0\}$ of strictly planar admissible maps. Consequently, no global minimizer of $\mathcal{G}_0$ in the full space $\mathcal{A}_{\boldsymbol{n}_b}(\Omega)$ can be planar: otherwise, it would also minimize $\mathcal{G}_0$ within $\mathcal{A}_{\mathrm{pl}}$ and, by uniqueness, coincide with $\boldsymbol{n}_{\mathrm{pl}}$, which is unstable for $\Theta > \Theta_c$.

\begin{theorem}[Bifurcation of accumulation points]
  \label{thm:bif}
  Assume $\Theta > \Theta_c$. There exist a family of exact physical minimizers $\{\boldsymbol{m}^{\varepsilon}\} \subset \operatorname{argmin} \mathcal{E}_{\varepsilon}$ and a sequence of local approximate minimizers $\{\boldsymbol{n}^{\varepsilon}\} \subset \operatorname{argmin} \mathcal{G}_{\varepsilon}$ such that, as $\varepsilon \to 0$,
  \begin{equation}
    \boldsymbol{m}^{\varepsilon} \to \boldsymbol{m} \quad \text{and} \quad \boldsymbol{n}^{\varepsilon} \to \boldsymbol{n}_* \qquad \text{strongly in } H^1(\Omega), \label{eq:existencelimits}
  \end{equation}
  where $\boldsymbol{m}, \boldsymbol{n}_* \in \operatorname{argmin} \mathcal{G}_0$ are macroscopically distinct limiting ground states, i.e., $\boldsymbol{m} \neq \boldsymbol{n}_*$.
\end{theorem}

\begin{proof}
  We divide the proof into three steps.
  
  \medskip
  
  \noindent \textit{Step 1. General strategy.} By the $\Gamma$-convergence result established in Section~\ref{ssec:gamma}, there exist families $\{\boldsymbol{m}^{\varepsilon}\} \subset \operatorname{argmin} \mathcal{E}_{\varepsilon}$ and $\{\boldsymbol{n}^{\varepsilon}\} \subset \operatorname{argmin} \mathcal{G}_{\varepsilon}$ such that
  \begin{equation}
    \boldsymbol{m}^{\varepsilon} \to \boldsymbol{m} \quad \text{and} \quad \boldsymbol{n}^{\varepsilon} \to \boldsymbol{n} \qquad \text{strongly in } H^1(\Omega),
  \end{equation}
  for some $\boldsymbol{m}, \boldsymbol{n} \in \operatorname{argmin} \mathcal{G}_0$. If $\boldsymbol{m} \neq \boldsymbol{n}$, the proof is complete upon setting $\boldsymbol{n}_* := \boldsymbol{n}$.
  
  Conversely, if $\boldsymbol{m} = \boldsymbol{n}$, we exploit the symmetry of the local functional. We set $\boldsymbol{n}_*^{\varepsilon} := \sigma(\boldsymbol{n}^{\varepsilon})$, where $\sigma : H^1(\Omega; \mathbb{S}^2) \to H^1(\Omega; \mathbb{S}^2)$ is the spatial reflection operator defined by $\sigma(n_1, n_2, n_3) = (n_1, n_2, -n_3)$. Because $\mathcal{G}_{\varepsilon}$ is perfectly invariant under reflection, $\boldsymbol{n}_*^{\varepsilon}$ remains a valid sequence of global minimizers. By the continuity of the operator $\sigma$ in $H^1(\Omega)$, we clearly have
  \begin{equation}
    \boldsymbol{n}_*^{\varepsilon} \to \sigma(\boldsymbol{n}).
  \end{equation}
  Setting $\boldsymbol{n}_* := \sigma(\boldsymbol{n})$, the result follows as soon as we establish that $\sigma(\boldsymbol{n}) \neq \boldsymbol{n}$ for $\Theta > \Theta_c$.
  
  \medskip
  
  \noindent \textit{Step 2. Instability of the planar state.} If the imposed boundary twist exceeds the critical threshold ($\Theta > \Theta_c$), the planar state $\boldsymbol{n}_{\mathrm{pl}}$ becomes energetically unstable. Consequently, no global minimizer of $\mathcal{G}_0$ can be planar.
  
  The second variation of the Dirichlet energy evaluated at $\boldsymbol{n}_{\mathrm{pl}}$ in the direction of a valid tangent variation $\boldsymbol{v} \in H^1_0(\Omega; \mathbb{R}^3)$ (satisfying $\boldsymbol{v} \cdot \boldsymbol{n}_{\mathrm{pl}} = 0$ almost everywhere) is given by:
  \begin{equation}
    \delta^2 \mathcal{G}_0(\boldsymbol{n}_{\mathrm{pl}})[\boldsymbol{v}] = 2 \int_{\Omega} \left( |\nabla \boldsymbol{v}|^2 - |\nabla \boldsymbol{n}_{\mathrm{pl}}|^2 |\boldsymbol{v}|^2 \right) dx.
  \end{equation}
  Any strictly out-of-plane vector field $\boldsymbol{v} = \eta \boldsymbol{e}_z$, for a scalar field $\eta \in H^1_0(\Omega)$, is orthogonal to $\boldsymbol{n}_{\mathrm{pl}}$ and therefore constitutes an admissible tangent variation. Substituting the gradient norm $|\nabla \boldsymbol{n}_{\mathrm{pl}}|^2 = (\phi')^2 = \Theta^2 / L^2$ yields the quadratic form:
  \begin{equation}
    \delta^2 \mathcal{G}_0(\boldsymbol{n}_{\mathrm{pl}})[\eta \boldsymbol{e}_z] = 2 \int_{\Omega} \left( |\nabla \eta|^2 - \frac{\Theta^2}{L^2} \eta^2 \right) dx.
  \end{equation}
  Testing this form against the principal Dirichlet eigenfunction $\eta_1$ of the domain evaluates exactly to $2(\lambda_1(\Omega) - \Theta^2 / L^2) \|\eta_1\|_{L^2}^2$. By definition, the critical threshold corresponds to the first Dirichlet eigenvalue: $\lambda_1(\Omega) = \Theta_c^2 / L^2$. Therefore, the second variation is strictly negative if and only if $\Theta > \Theta_c$.
  
  A negative second variation guarantees that the state is an unstable saddle point, implying the existence of admissible competitors with strictly lower energy. Consequently, the true global infimum strictly satisfies
  \begin{equation}
    \min_{\mathcal{A}_{\boldsymbol{n}_b}(\Omega)} \mathcal{G}_0 < \mathcal{G}_0(\boldsymbol{n}_{\mathrm{pl}}). \label{eq:planarhighen}
  \end{equation}
  On the other hand, it is standard to show that $\boldsymbol{n}_{\mathrm{pl}}$ is the unique global minimizer of the unperturbed energy $\mathcal{G}_0$ restricted to the class $\mathcal{A}_{\mathrm{pl}} = \{\boldsymbol{n} \in \mathcal{A}_{\boldsymbol{n}_b}(\Omega) : \boldsymbol{n} \cdot \boldsymbol{e}_z \equiv 0\}$ of strictly planar admissible maps. Consequently, no global minimizer of $\mathcal{G}_0$ in the full space $\mathcal{A}_{\boldsymbol{n}_b}(\Omega)$ can be planar: otherwise, it would also minimize $\mathcal{G}_0$ in $\mathcal{A}_{\mathrm{pl}}$ and, by uniqueness, coincide with $\boldsymbol{n}_{\mathrm{pl}}$, which contradicts \eqref{eq:planarhighen} for $\Theta > \Theta_c$.
  
  \medskip
  
  \noindent \textit{Step 3. Conclusion: $\sigma(\boldsymbol{n}) \neq \boldsymbol{n}$.} The previous step guarantees that the global minimizer $\boldsymbol{n}$ is not planar when $\Theta > \Theta_c$. Since the fixed points of the reflection operator $\sigma$ are strictly limited to the planar configurations, it follows immediately that $\sigma(\boldsymbol{n}) \neq \boldsymbol{n}$.
\end{proof}

\section*{Declarations}

\noindent
\textit{Author Contributions:} All authors contributed equally to this work.
\smallskip

\noindent
\textit{Ethical Approval:} Not applicable. This study does not involve human participants or animals.
\smallskip

\noindent
\textit{Conflict of Interest:} The authors declare that they have no conflicts of interest.
\smallskip

\noindent
\textit{Data and Code Availability:} No datasets were generated or analyzed during the current study. 
\smallskip

\noindent
\textit{Declaration of Generative AI and AI-Assisted Technologies:}
During the preparation of this manuscript, the authors used Grammarly and Gemini to assist with spell-checking and to improve the fluency of selected portions of the text. The manuscript was initially prepared in TeXmacs and subsequently exported to LaTeX. Gemini was also used to identify and correct conversion errors arising from this export process and to generate or refine LaTeX code for improving the document layout, typography, and overall presentation.

No generative AI tools were used in the development of the mathematical ideas, the derivation of the theoretical results, or the proofs. The authors are fully responsible for the content of the manuscript and for the accuracy, validity, and integrity of all results presented.

\section*{Acknowledgements}
\textsc{G.D.F.} is a member of GNAMPA--INdAM. He gratefully acknowledges partial financial support from the GNAMPA Project CUP\_E53C25002010001, and from the University of Naples Federico II through the FRA Project-B ``VarMoCry'' on  \emph{Variational Analysis and Modeling of Liquid Crystals}. Further support is acknowledged from the Italian Ministry of University and Research through the PRIN 2022 project \emph{Variational Analysis of Complex Systems in Material Science, Physics and Biology} (No.~2022HKBF5C).

 The work of AZ has been partially supported by the Basque Government through the BERC 2022-2025 program and by the Spanish State Research Agency through BCAM Severo Ochoa excellence accreditation Severo Ochoa CEX2021-00114 and through project PID2023-146764NB-I00
funded by MICIU/AEI/10.13039/501100011033. AZ was also partially supported by a grant of the Ministry of Research, Innovation and Digitization, CNCS-UEFISCDI, project number PN-IV-P2-2.1-T-TE-2023-0219, within PNCDI IV.

\bibliographystyle{alpha}
\bibliography{DFSZ_LC_Electric_2026}

\end{document}